\documentclass[letterpaper, 12pt]{amsproc}

\usepackage{amsmath,amssymb,anysize,times,float,enumerate}
\usepackage{listings}
\usepackage{multicol}

\usepackage{hyperref} 
\hypersetup{citecolor=red, linkcolor=blue, colorlinks=true}

\usepackage{tikz}
\tikzstyle{every node} = [inner sep=.1cm]

\renewcommand\le{\leqslant}
\renewcommand\ge{\geqslant}

\usetikzlibrary{decorations.markings}
\usetikzlibrary{arrows}

\tikzstyle{vertex}=[circle,draw,inner sep=0pt, minimum size=6pt]

\newtheorem{thm}{Theorem}[section]
\newtheorem{lem}[thm]{Lemma}
\newtheorem{prop}[thm]{Proposition}
\newtheorem{cor}[thm]{Corollary}
\newtheorem{const}[thm]{Construction}

\theoremstyle{definition}

\theoremstyle{remark}

\newcommand\Wr{{\rm \,wr\, }}

\newcommand{\Aut}{\operatorname{Aut}}
\newcommand{\GF}{\operatorname{GF}}
\newcommand{\Sym}{\operatorname{Sym}}
\def\CT{\operatorname{CT}}

\newcommand{\M}{\mathcal{M}}

\newcommand{\Z}{\mathbb{Z}}

\renewcommand\ell{l}

\title{Graphs that contain multiply transitive matchings}

\author{Alex Schaefer}
\address{Department of Mathematics, University of Kansas, 1460 Jayhawk Blvd, Lawrence, KS 66045}
\email{alex.scha4@ku.edu}

\author{Eric Swartz}
\address{Department of Mathematics, William \& Mary, P.O. Box 8795, Williamsburg, VA 23187-8795}
\email{easwartz@wm.edu}



\begin{document}

\begin{abstract}
Let $\Gamma$ be a finite, undirected, connected, simple graph.  We say that a matching $\mathcal{M}$ is a \textit{permutable $m$-matching} if $\mathcal{M}$ contains $m$ edges and the subgroup of $\text{Aut}(\Gamma)$ that fixes the matching $\mathcal{M}$ setwise allows the edges of $\mathcal{M}$ to be permuted in any fashion.  A matching $\mathcal{M}$ is \textit{2-transitive} if the setwise stabilizer of $\mathcal{M}$ in $\text{Aut}(\Gamma)$ can map any ordered pair of distinct edges of $\mathcal{M}$ to any other ordered pair of distinct edges of $\mathcal{M}$.  We provide constructions of graphs with a permutable matching; we show that, if $\Gamma$ is an arc-transitive graph that contains a permutable $m$-matching for $m \ge 4$, then the degree of $\Gamma$ is at least $m$; and, when $m$ is sufficiently large, we characterize the locally primitive, arc-transitive graphs of degree $m$ that contain a permutable $m$-matching.  Finally, we classify the graphs that have a $2$-transitive perfect matching and also classify graphs that have a permutable perfect matching.
\end{abstract}

\maketitle

\section{Introduction}

All graphs considered in this paper are finite, undirected, and simple, and are connected unless otherwise stated.  A \textit{matching} $\M$ is a set of edges of a graph $\Gamma$ such that no two are incident with a common vertex.  A matching $\M$ is a \textit{perfect matching} of $\Gamma$ if each vertex of $\Gamma$ is incident with exactly one edge of $\M$. In other words, a matching $\M$ is the edge set of a $1$-regular subgraph of $\Gamma$, and $\M$ is perfect exactly when the $1$-regular subgraph is spanning.  Let $\Gamma$ be a graph, let $\M$ be a matching in $\Gamma$ with $m$ edges, and let $G$ be a subgroup of $\Aut(\Gamma)$.  We will say that $\M$ is a \emph {$G$-permutable $m$-matching} if the restriction of the action of $G$ to the edge set of $\M$ is that of the symmetric group $S_{m}$ , i.e., if $G_\M^{E(\M)}\cong S_{m}$.  If such a group $G$ and matching $\M$ exist, we will say that the graph $\Gamma$ contains a \textit{permutable $m$-matching}.  The concept of a permutable matching is due to Zaslavsky, motivated by a question involving signed graphs from \cite{schaefercyclespace}.

A group $G$ of permutations of a set $\Omega$ is \textit{2-transitive} on $\Omega$ if, given two ordered pairs of distinct elements $(\alpha, \beta), (\gamma, \delta) \in \Omega \times \Omega$, there exists $g \in G$ such that $(\alpha, \beta)^g := (\alpha^g, \beta^g) = (\gamma, \delta);$
in other words, $G$ can map any ordered pair of distinct elements to any other ordered pair of distinct elements.  We say that a matching $\M$ of a graph $\Gamma$ is a \textit{$2$-transitive matching} if the setwise stabilizer of $\M$ in $\Aut(\Gamma)$ is $2$-transitive on the edges of $\M$.  

The purpose of this paper is to study graphs that contain a matching $\M$ such that the setwise stabilizer of $\M$ is multiply transitive on the edges of $\M$.  This paper is structured as follows.  In Section \ref{sect:background}, we provide background information necessary for the later sections.  In Section \ref{sect:constructions}, we provide various constructions for graphs with a permutable $m$-matching, showing that there is actually an abundance of such graphs for any $m$.  Moreover, there are even numerous examples when the graph $\Gamma$ is required to be \textit{$G$-arc-transitive} for $G \le \Aut(\Gamma)$, that is, when $G$ is transitive on the set $A(\Gamma)$ of ordered pairs of adjacent vertices.  In Section \ref{sect:local}, we prove the following result, which shows that the degree of a vertex cannot be too small in a graph with a permutable matching, up to a single, known family of exceptions.    

\begin{thm}
\label{thm:degreem}
 Let $G \le \Aut(\Gamma)$.  If $\Gamma$ is a connected $G$-arc-transitive graph with a $G$-permutable $m$-matching, then the degree of the graph $\Gamma$ is at least $m$ unless $m = 3$ and $\Gamma$ is the cycle $C_{3k}$, where $k \ge 2$.
\end{thm}

Many of the graphs with permutable matchings constructed in Section \ref{sect:constructions} contain a system of \textit{imprimitivity}, i.e., the full automorphism group of the graph preserves a nontrivial partition of the vertex set (and, in some cases, the stabilizer of a vertex $\alpha$ even preserves a nontrivial partition of the neighbors of $\alpha$).  If a group $G$ of permutations of a set $\Omega$ is transitive on $\Omega$ but $G$ does not preserve any partition of $\Omega$ other than the trivial partitions of $\Omega$ into singleton sets and the single set $\Omega$, then $G$ is \textit{primitive} on $\Omega$.  Given a graph $\Gamma$ and $G \le \Aut(\Gamma)$, $\Gamma$ is said to be \textit{$G$-locally primitive} if, given any $\alpha \in V(\Gamma)$, the stabilizer of $\alpha$ in $G$ is primitive on the neighbors of $\alpha$.  Given the constructions in Section \ref{sect:constructions} and Theorem \ref{thm:degreem}, it makes sense to consider graphs with degree $m$ that are locally primitive and arc-transitive containing a permutable $m$-matching.  In Section \ref{sect:degreem}, we provide a characterization of such graphs.  The notation and terminology used in the following theorem are explained in depth in Section \ref{sect:background}.

\begin{thm}
\label{thm:mvalent}
Let $\Gamma$ be a connected $G$-arc-transitive, $G$-locally primitive graph with degree $m \ge 6$ that contains a $G$-permutable $m$-matching, and suppose $G$ has a nontrivial normal subgroup $N$ that has more than two orbits on vertices.  If the normal quotient graph $\Gamma_N$ does not contain a permutable $m$-matching, then $\Gamma_N$ is a near-polygonal graph and $(\Gamma_N, G/N)$ is locally-$S_m$.
\end{thm}

A group $G$ is said to be \textit{quasiprimitive} on a set $\Omega$ if every nontrivial normal subgroup of $G$ is transitive on $\Omega$, and a group $G$ is said to be \textit{biquasiprimitive} on a set $\Omega$ if $\Omega$ has a $G$-invariant partition $\Omega = \Delta_1 \cup \Delta_2$ such that the setwise stabilizer $G_{\Delta_i}$ is quasiprimitive on $\Delta_i$ for $i = 1,2$.  Using this terminology, Theorem \ref{thm:mvalent} says that, if there exists a graph $\Gamma$ that is $G$-arc-transitive and $G$-locally primitive with degree $m \ge 6$ that contains a $G$-permutable $m$-matching, then one can keep taking normal quotients of this graph until reaching either (1) a vertex-quasiprimitive graph with a permutable $m$-matching, (2) a vertex-biquasiprimitive graph with a permutable $m$-matching, or (3) a near-polygonal graph such that the stabilizer of a vertex can permute the $m$ neighbors in any way; see Section \ref{sect:background}.  Moreover, graphs in each case exist and are constructed in Section \ref{sect:constructions}. We do not know if the theorem holds for $m\leq5$; the restriction on $m$ is a result of the technique.

Section \ref{sect:classification} is devoted to the proof of the following theorem, which classifies the graphs with a $2$-transitive perfect matching. Joins and matching joins are defined following the statement of the theorem.

\begin{thm}
\label{thm:2transperfect}
 Let $\Gamma$ be a connected graph on $2m$ vertices with a $2$-transitive perfect matching $\M$ containing $m$ edges.  Then we have one of the following cases:
 \begin{itemize}
  \item[(1)] $\Gamma$ is join between two graphs that are either complete or edgeless:
    \begin{itemize}
      \item[(a)] $K_m \vee K_m \cong K_{2m}$,
      \item[(b)] $K_m \vee \overline{K}_m$,
      \item[(c)] $\overline{K}_m \vee \overline{K}_m \cong K_{m,m}$.
    \end{itemize}
  \item[(2)] $\Gamma$ is a matching join between two graphs that are either complete or edgeless (but not both edgeless):
    \begin{itemize}
      \item[(a)] $K_m \veebar  K_m$,
      \item[(b)] $K_m \veebar  \overline{K}_m$.
    \end{itemize}
  \item[(3)] Let $m = p^f$, where $p$ is a prime and $p^f \equiv 3 \pmod 4$.  Then either:
    \begin{itemize}
      \item[(a)] $\Gamma$ is the incidence graph of the Paley symmetric $2$-design over $\GF(p^f)$, i.e., $V(\Gamma) = \GF(p^f) \times \{0,1\}$, and $(x,i), (y,j) \in V(\Gamma)$ are adjacent if and only if $i = 0$, $j = 1$, and $y -x$ is a square in $\GF(p^f)$; or
      \item[(b)] $\Gamma$ is the graph obtained by taking the incidence graph of the Paley symmetric $2$-design over $\GF(p^f)$ and replacing the independent sets with copies of $K_{p^f}$; that is $V(\Gamma) = \GF(p^f) \times \{0,1\}$, and $(x,i), (y,j) \in V(\Gamma)$ are adjacent if and only if either $i = j$ and $x \neq y$ or if $i = 0$, $j = 1$, and $y -x$ is a square in $\GF(p^f)$.
    \end{itemize}
  \item[(4)] Let $m = 5$.  Then either
    \begin{itemize}
      \item[(a)] $\Gamma$ is the Petersen graph; or
      \item[(b)] $\Gamma = C_5 \vee C_5$.
    \end{itemize}

 \end{itemize}

\end{thm}

Here, $\Gamma_1 \vee \Gamma_2$ denotes the \textit{join} of the graphs $\Gamma_1$ and $\Gamma_2$, in which $V(\Gamma_1 \vee \Gamma_2) = V(\Gamma_1) \cup V(\Gamma_2)$ and 
\[E(\Gamma_1 \vee \Gamma_2) = E(\Gamma_1) \cup E(\Gamma_2) \cup \{\{\alpha, \beta\} : \alpha \in V(\Gamma_1), \beta \in V(\Gamma_2)\}.\]

The notation $\Gamma_1 \veebar _{\phi} \Gamma_2$ denotes a \textit{matching join} of $\Gamma_1$ and $\Gamma_2$.  In this case, both $\Gamma_1$ and $\Gamma_2$ must be graphs with $|V(\Gamma_1)| = |V(\Gamma_2)|$ and $\phi: V(\Gamma_1) \rightarrow V(\Gamma_2)$ is a bijection between the vertex sets. The graph $\Gamma_1 \veebar _{\phi} \Gamma_2$ is defined to have vertex set $V(\Gamma_1 \veebar _{\phi} \Gamma_2) = V(\Gamma_1) \cup V(\Gamma_2)$ and the edge set is
\[E(\Gamma_1 \veebar _{\phi} \Gamma_2) = E(\Gamma_1) \cup E(\Gamma_2) \cup \{\{\alpha, \alpha^{\phi}\} : \alpha \in V(\Gamma_1)\}.\]
When $\Gamma_1$ or $\Gamma_2$ is a complete graph or an empty graph, then the resulting graph is unique up to isomorphism regardless of the choice of $\phi$, and in this case we simply use the notation $\Gamma_1 \veebar  \Gamma_2$.

As an example of a matching join, consider two copies of $C_5$: $\Gamma_1 = \{1,2,3,4,5\}$ with $x$ adjacent to $y$ if and only if $x - y \equiv \pm 1 \pmod 5$ and $\Gamma_2 = \{6,7,8,9,10\}$ with, again, $x$ adjacent to $y$ if and only if $x - y \equiv \pm 1 \pmod 5$.  If we define $\phi$ to be $x^{\phi} = x + 5$, then the matching join $\Gamma_1 \veebar _{\phi} \Gamma_2$ is isomorphic to the $5$-prism, whereas if we define $\phi: \Gamma_1 \rightarrow \Gamma_2$ by $1^{\phi} = 6$, $2^{\phi} = 9$, $3^{\phi} = 7$, $4^{\phi} = 10$, and $5^{\phi} = 8$, then the matching join $\Gamma_1 \veebar _{\phi} \Gamma_2$ is isomorphic to the Petersen graph.

As a corollary of Theorem \ref{thm:2transperfect} we classify all connected graphs with a permutable perfect matching.

\begin{cor}
 \label{cor:mpermperfect}
 Let $\Gamma$ be a connected graph on $2m$ vertices with a permutable perfect matching $\M$.  Then $\Gamma$ is one of $K_{2m}$, $K_m \vee \overline{K}_m$, $K_{m,m}$, $K_m \veebar  K_m$, $K_m \veebar  \overline{K}_m$, $C_6$, or $K_6 \backslash \{3\cdot K_2\} \cong K_{2,2,2}$.
\end{cor}

In particular, Theorem \ref{thm:2transperfect} classifies the possible induced subgraphs on the vertex set of a $2$-transitive matching $\M$ of size $m$ in an arbitrary graph: either the induced subgraph is disconnected and is $m \cdot K_2$ (i.e., $m$ vertex-disjoint edges) or it is connected and is one of the graphs listed in Theorem \ref{thm:2transperfect}. Moreover, the induced subgraph on the vertex set of a permutable $m$-matching $\M$ in an arbitrary graph is either $m \cdot K_2$ or one of the graphs listed in Corollary \ref{cor:mpermperfect}.
%
%

\section{Background}
\label{sect:background}

In this section we review the terminology and theory that will be used in later sections.

Let $\Gamma$ be a graph.  Given a subset $X$ of the vertices of $\Gamma$, the induced subgraph of $\Gamma$ on $X$ is denoted by $\Gamma[X]$.  We denote the fact that the vertices $\alpha$ and $\beta$ are adjacent by writing $\alpha \sim \beta$. We denote by $\overline{\Gamma}$ the complement of $\Gamma$.  A \textit{walk} $W$ is defined to be a sequence of vertices $(\alpha_0, \alpha_1, \dots, \alpha_n)$ such that $\alpha_i \sim \alpha_{i+1}$ for $0 \le i \le n -1$.  For $\alpha \in V(\Gamma)$, we denote the set of neighbors of $\alpha$ in $\Gamma$ by $\Gamma(\alpha)$.  The \textit{degree} of a vertex $\alpha$ is $|\Gamma(\alpha)|$, and we say that the graph $\Gamma$ is \textit{regular} if every vertex has the same degree.

A graph $\Gamma$ is said to be $(v,k,\lambda, \mu)$-strongly regular if $\Gamma$ has $v$ vertices; $\Gamma$ is regular of degree $k$; if $\alpha, \beta \in V(\Gamma)$ and $\alpha \sim \beta$, then $|\Gamma(\alpha) \cap \Gamma(\beta)| = \lambda$; and if $\alpha \neq \beta \in V(\Gamma)$ and $\alpha \not\sim \beta$, then $|\Gamma(\alpha) \cap \Gamma(\beta)| = \mu$. 

\subsection{Permutation groups and graph symmetry}
\label{subsect:permgps}

Let $\Omega$ be a set and $G$ a group of permutations of $\Omega$, that is, let $G \le \Sym(\Omega)$.  For an element $\omega \in \Omega$, the orbit of $\omega$ under $G$ is denoted by $\omega^G$.  For a subset $\Delta$ of $\Omega$, we let $G_\Delta$ denote the setwise stabilizer of $\Delta$ in $G$.  When $\Delta = \{\omega\}$, a single element of $\Omega$, we write $G_\omega := G_{\{\omega\}}$.  If $\Delta = \{\omega_1, \omega_2, \dots, \omega_k\}$, then \[G_{\omega_1\omega_2 \dots \omega_k} := \bigcap_{i=1}^k \limits G_{\omega_i},\] that is, $G_{\omega_1\omega_2 \dots \omega_k}$ fixes every $\omega_i$.  For instance, $G_{\alpha\Delta}$ denotes $G_{\{\alpha\}}\cap G_{\Delta}$, i.e. the set of elements of $G$ which stabilize both the element $\alpha$ pointwise and the set $\Delta$ setwise.  If $H \le G_\Delta$, then we denote by $H^\Delta$ the induced action of $H$ on $\Delta$, i.e., $H^\Delta$ is the image of the natural homomorphism from $H$ into $\Sym(\Delta)$.  If $G$ is a group of permutations of $\Omega_1$ and $G^\prime$ is a group of permutations of $\Omega_2$, then $G$ and $G^\prime$ are said to be \textit{permutation isomorphic} if there are both a bijection $\psi : \Omega_1 \rightarrow \Omega_2$ and a group isomorphism $\phi: G \rightarrow G^\prime$ such that, for all $g \in G$ and $\omega \in \Omega_1$, $(\omega^g)^\psi = (\omega^\psi)^{g^\phi}$.

The group of permutations $G$ is said to be \textit{transitive} on $\Omega$ if, for every $\alpha, \beta \in \Omega$, there exists $g \in G$ such that $\alpha^g = \beta.$  A group $G$ of permutations of a set $\Omega$ is said to be \textit{regular} on $\Omega$ if $G$ is transitive on $\Omega$ and $G_\omega = 1$ for all $\omega \in \Omega$.  Additionally, $G$ is said to be \textit{primitive} on $\Omega$ if $G$ is transitive on $\Omega$ and $G$ preserves no nontrivial partition of $\Omega$, that is, $G$ preserves no partition of $\Omega$ other than the partition into singleton sets and the partition into the single set $\Omega$.  If $\Pi$ is a nontrivial $G$-invariant partition of $\Omega$, then $\Pi$ is called a \textit{system of imprimitivity} and the elements of $\Pi$ are called \textit{blocks}.  Finally, a group $G$ is said to be \textit{biprimitive} on $\Omega$ if $\Omega$ has a $G$-invariant partition $\Omega = \Delta_1 \cup \Delta_2$ such that the setwise stabilizer $G_{\Delta_i}$ is primitive on $\Delta_i$ for $i = 1,2$.

The group of permutations $G$ is said to be \textit{quasiprimitive} on the set $\Omega$ if every nontrivial normal subgroup of $G$ is transitive on $\Omega$.  If $G$ is primitive on $\Omega$, then $G$ is quasiprimitive on $\Omega$; however, the converse is not true.  A group $G$ is said to be \textit{biquasiprimitive} on $\Omega$ if $\Omega$ has a $G$-invariant partition $\Omega = \Delta_1 \cup \Delta_2$ such that the setwise stabilizer $G_{\Delta_i}$ is quasiprimitive on $\Delta_i$ for $i = 1,2$.

Let $\Gamma$ be a graph with vertex set $V(\Gamma)$ and edge set $E(\Gamma)$.  An \textit{automorphism} of a graph $\Gamma$ is a permutation of the vertices that preserves adjacency.  The set of automorphisms of $\Gamma$ forms a group, which is denoted by $\Aut(\Gamma)$.  Note that $\Aut(\Gamma) \le \Sym(V(\Gamma))$.

Let $G \le \Aut(\Gamma)$.  The graph $\Gamma$ is \textit{$G$-vertex-transitive} if $G$ is transitive on the vertices of $\Gamma$, and $\Gamma$ is \textit{$G$-edge-transitive} if $G$ is transitive on edges.  Similarly, the graph $\Gamma$ is \textit{$G$-vertex-quasiprimitive} (respectively, \textit{$G$-vertex-biquasiprimitive}) if $G$ is quasiprimitive (respectively, biquasiprimitive) on the vertices of $\Gamma$.  An \textit{arc} is an ordered pair of vertices $(\alpha, \beta)$ such that $\{\alpha, \beta\} \in E(\Gamma)$, and $\Gamma$ is \textit{$G$-arc-transitive} if $G$ is transitive on the set $A(\Gamma)$ of arcs of $\Gamma$.  More generally, an \textit{$s$-arc} of $\Gamma$ is an ordered $(s+1)$-tuple of vertices $(\alpha_0, \dots, \alpha_s)$ such that $\{\alpha_i, \alpha_{i+1} \} \in E(\Gamma)$ for $0 \le i \le s-1$ and $\alpha_{j-1} \neq \alpha_{j+1}$ for $1 \le j \le s-1$.  (Repeated vertices are allowed in the walk defined by the $s$-arc, but there are no returns in the walk.)  The graph $\Gamma$ is said to be \textit{$(G,s)$-arc-transitive} if $G$ is transitive on the set of $s$-arcs of $\Gamma$. 

Given vertices $\alpha, \beta$ of $\Gamma$, we define the \textit{distance} between $\alpha$ and $\beta$ to be the length of a shortest path between $\alpha$ and $\beta$ (measured in edges), and we denote the distance between $\alpha$ and $\beta$ by $d(\alpha, \beta)$.  Since we are only considering connected graphs, there will always exist a path between any two vertices $\alpha$ and $\beta$, so distance is a well-defined, finite-valued function on pairs of vertices.  Given a fixed vertex $\alpha$, for every natural number $i$ we let
\[G_\alpha^{[i]} := \{g \in G_\alpha : \beta^g = \beta \text{ for all } \beta \in V(\Gamma) \text{ such that } d(\alpha, \beta) \le i\},\]
that is, $G_\alpha^{[i]}$ is the group that fixes pointwise the set of all vertices at distance at most $i$ from $\alpha$.  In particular, 
\[G_\alpha^{[1]} = \{g \in G_\alpha : \beta^g = \beta \text{ for all } \beta \in \Gamma(\alpha)\}, \]
and $G_\alpha^{[1]}$ is often referred to as the \textit{kernel of the local action of $G$} since, for the induced action $G_\alpha^{\Gamma(\alpha)}$ of the vertex stabilizer $G_\alpha$ on the neighbors of $\alpha$, we have $G_\alpha^{\Gamma(\alpha)} \cong G_\alpha/G_\alpha^{[1]}.$  Finally, for vertices $\alpha_1, \alpha_2, \dots, \alpha_k$, we define \[G_{\alpha_1 \dots \alpha_k}^{[1]}:= \bigcap_{i=1}^k \limits G_{\alpha_i}^{[1]},\] that is, $G_{\alpha_1 \dots \alpha_k}^{[1]}$ is the pointwise stabilizer of the union of the $\Gamma(\alpha_{i})$.

Given a permutation group $L$, a graph $\Gamma$, $\alpha \in V(\Gamma)$, and $G \le \Aut(\Gamma)$ such that $\Gamma$ is $G$-vertex-transitive, the pair $(\Gamma, G)$ is said to be \textit{locally-$L$} if $G_{\alpha}^{\Gamma(\alpha)}$ is permutation isomorphic to $L$.  The graph $\Gamma$ is said to be \textit{$G$-locally primitive} if $G_\alpha^{\Gamma(\alpha)}$ is primitive on $\Gamma(\alpha)$.

\subsection{Normal quotient graphs, voltage graphs, and regular covers}
\label{subsect:quotientcover}


Let $\Gamma$ be a graph with transitive group of automorphisms $G$, and let $N$ be an intransitive normal subgroup of $G$.  The $N$-orbits of vertices of $\Gamma$ form a system of imprimitivity for $G$, and the \textit{normal quotient graph} $\Gamma_N$ with respect to the normal subgroup $N$ is the graph whose vertex set is the $N$-orbits of vertices, and two $N$-orbits $\alpha^N$ and $\beta^N$ are adjacent if and only if there is $\alpha^\prime \in \alpha^N$ and $\beta^\prime \in \beta^N$ such that $\alpha^\prime \sim \beta^\prime$.  The graph $\Gamma$ is said to be a \textit{regular cover} of $\Gamma_N$ if, given any two adjacent vertices $\alpha^N$ and $\beta^N$ in $\Gamma_N$, we have $|\Gamma(\alpha) \cap \beta^N| = 1$.

The following lemma is a well-known result, and it shows that local primitivity is a sufficient condition for the original graph to be a regular cover of the normal quotient graph.

\begin{lem}\cite[Theorem 10.4]{PraegerLiNiemeyer}
 Let $\Gamma$ be a $G$-vertex-transitive and $G$-locally primitive graph, where $G \le \Aut(\Gamma)$, and let $N$ be a normal subgroup of $G$ with more than two orbits on $V(\Gamma)$.  Then $\Gamma$ is a regular cover of the quotient graph $\Gamma_N$, and the quotient graph $\Gamma_N$ is $G/N$-vertex-transitive and $G/N$-locally primitive.
\end{lem}

An equivalent definition of a regular cover is as follows.  A \textit{covering projection} $p: \tilde{\Gamma} \rightarrow \Gamma$ maps $V(\tilde{\Gamma})$ onto $V(\Gamma)$, preserving adjacency, such that for any vertex $\tilde{\alpha} \in V(\tilde{\Gamma})$, the set of neighbors of $\tilde{\alpha}$ is mapped bijectively onto the set of neighbors of $\tilde{\alpha}^p$.  For a vertex $\alpha$ of $\Gamma$, the set $\alpha^{p^{-1}}$ of vertices that are mapped onto $\alpha$ by $p$ is called the \textit{fiber} over the vertex $\alpha$.  An automorphism $g \in \Aut(\Gamma)$ \textit{lifts} to $\tilde{g} \in \Aut(\tilde{\Gamma})$ if the following diagram commutes:
\begin{center}
\begin{tikzpicture}
	\node (1) at (0,2){$\tilde{\Gamma}$};
	\node (2) at (2,2) {$\tilde{\Gamma}$};
	\node (3) at (0,0) {$\Gamma$};
	\node (4) at (2,0) {$\Gamma$};
	\draw[->] (1) to node [anchor=south] {$\tilde{g}$} (2);
 	\draw[->] (1) to node [anchor=east] {$p$} (3);
	\draw[->] (2) to node [anchor=west] {$p$} (4);
	\draw[->] (3) to node [anchor=north] {$g$} (4);
\end{tikzpicture}
\end{center}
The lift of the trivial group (identity) is known as the group of \textit{covering transformations} and is denoted $\CT(p)$.  The graph $\tilde{\Gamma}$ is a \textit{regular cover} of $\Gamma$ if $\CT(p)$ acts regularly on the set $\alpha^{p^{-1}}$ for all vertices $\alpha \in V(\Gamma).$

A \textit{voltage assignment} on a graph $\Gamma$ is a map $\xi:A(\Gamma) \rightarrow H$, where $H$ is a group, such that $(\alpha, \beta)^\xi = \left((\beta,\alpha)^\xi\right)^{-1}$, and a \textit{voltage graph} is a graph $\Gamma$ together with a voltage assignment.  For ease of notation, the voltage of the arc $(\alpha,\beta)$ will be denoted $\xi_{\alpha\beta}$, and $\xi_W$ will denote the total voltage of a walk $W$, that is, $\xi_W$ is the product (or sum, depending on the group operation) of the voltages of the edges in $W$.  The \textit{derived covering graph} $\tilde{\Gamma}$ of a voltage graph has vertex set $V(\Gamma) \times H$, where two vertices $(\alpha, h_1)$ and $(\beta, h_2)$ are adjacent iff $\alpha$ is adjacent to $\beta$ in $\Gamma$ and $h_2 = \xi_{\alpha\beta}h_1.$  The following theorem exhibits the deep connection between regular covers and derived covering graphs:

\begin{lem}[{\cite[Theorem 2.4.5, Section 2.5]{topol}}]
\label{thm:grosstucker}
Every regular cover $\tilde{\Gamma}$ of a graph $\Gamma$ is a derived cover of a voltage graph (and conversely).

In addition, suppose the voltage group is generated by the voltages assigned to the edges of $\Gamma$. If the edges of a (fixed but arbitrary) spanning tree of $\Gamma$ have the identity voltage, then $\tilde{\Gamma}$ is connected.
\end{lem}

%

Fix a spanning tree $\mathcal{T}$ of a graph $\Gamma$. Choose $\alpha\in V(\Gamma)$, and assume that the edges of $\mathcal{T}$ have been assigned the identity voltage. This implies that the voltage assignment $\xi$ induces a natural homomorphism of the fundamental group of $\Gamma$ based at $\alpha$ (generated by all closed walks in $\Gamma$ based at $\alpha$) into the voltage group $H$. Let $g\in\Aut(\Gamma)$. For each closed walk $W$ based at $\alpha$, $W^{g}$ will be a closed walk based at $\alpha^{g}$. Moreover, the walk formed by the path in $\mathcal{T}$ from $\alpha$ to $\alpha^{g}$, followed by $W^{g}$, followed by the path in $\mathcal{T}$ from $\alpha^{g}$ back to $\alpha$, is a closed walk based at $\alpha$ with the same voltage as $W^{g}$. This induces a multivalued function $g^{\phi_{\alpha}}:H\to H$ given by $(\xi_W)^{g^{\phi_\alpha}} := \xi_{W^{g} }$.  This is not necessarily well-defined, as two walks $W_{1}$ and $W_{2}$ may have the same voltage while $W_{1}^{g}$ and $W_{2}^{g}$ may not. Furthermore, $g^{\phi_{\alpha}}$ may not be defined on all of $H$. With this in mind, the following lemma gives explicit criteria for an automorphism of a graph to lift.

\begin{lem}[{\cite[Propositions 3.1, 5.1]{elab}}]
\label{lem:lift} 
Fix a spanning tree $\mathcal{T}$ of a graph $\Gamma$ and $\alpha\in V(\Gamma)$. Assume the edges of $\mathcal{T}$ are assigned the identity voltage and that the voltage group $H$ is generated by the edge voltages of $\Gamma$. An automorphism $g$ of $\Gamma$ lifts to an automorphism $\tilde{g}$ of $\tilde{\Gamma}$ if and only if $g^{\phi_\alpha}$ is a group automorphism. Moreover, if $H$ is abelian, the automorphism $g^{\phi_\alpha}$ does not depend on the choice of base vertex $\alpha$.
\end{lem}

The following lemma also shows that it is quite possible to get the entire automorphism group of a graph to lift.

\begin{lem}[{\cite[Proposition 6.4, Theorem 5.2]{elab}}]
\label{lem:allautlifts}
Let $\Gamma$ be a graph with edge set $E$, and let $T$ denote the set of edges of a spanning tree $\mathcal{T}$ of $\Gamma$.   Let $\Z_p$ denote the cyclic group of order $p$, where $p$ is a prime. Let $H:= \Z_p^{|E| - |T|}$; $H$ is a $\Z_{p}$-vector space. Let $X$ be a basis for $H$, so $|X|=|E| - |T|$.  Define $\Gamma_p$ to be the derived regular cover of the voltage graph defined by assigning a distinct element of $X$ to each co-tree edge of $\Gamma$.  Then $\Gamma_p$ is well-defined, unique up to graph isomorphism, and $\Aut(\Gamma)$ lifts.  Moreover, the induced mapping $\phi: \Aut(\Gamma) \rightarrow \Aut(H)$ is a group homomorphism.
\end{lem}

\subsection{Near-polygonal graphs}
\label{subsect:nearpolygonal}

Following \cite{perkel:near}, we say that $\Gamma$ is a \textit{near-polygonal graph} if there exists a distinguished set of $c$-cycles $\mathcal{C}$ such that every $2$-path of $\Gamma$ is contained in a unique cycle in $\mathcal{C}$.  If $c$ is the girth of $\Gamma$, then $\Gamma$ is called a \textit{polygonal graph}.  Furthermore, if our collection $\mathcal{C}$ of $c$-cycles is in fact the set of all cycles of length $\text{girth}(\Gamma)$, then $\Gamma$ is called \textit{strict polygonal.}  Manley Perkel invented the notion of a polygonal graph in \cite{perkel:thesis} and that of a near-polygonal graph in \cite{perkel:near}.  (Perkel's original definition of near-polygonal graphs required that the length $c$ of the special cycles be greater than $3$. In our definition, we allow $c=3$.)  

Polygonal graphs are a natural generalization of the edge- and vertex-set of polygons and some Platonic solids (such as the cube and dodecahedron), and one immediately notes that these are themselves strict polygonal graphs, with the special set of cycles being the polygon itself or the faces of the solid, respectively.  The complete graph on $n$ points, $K_n$, is a strict polygonal graph of girth 3, and the Petersen graph is a polygonal graph of girth 5 that is not a strict polygonal graph \cite{seress:survey}.  Very few examples of polygonal graphs are known; see \cite{nearpolyg10, swartz:odd, swartz:even}.

Near-polygonal graphs have appeared in the past when studying quotient graphs of symmetric graphs \cite{ZhouAlmostCovers, ZhouAlmostCoversErratum}.  We mention here the following result, which gives a sufficient condition for a graph $\Gamma$ to be near-polygonal:

\begin{lem}[{\cite[Theorem 1]{ZhouNearPolyg}}]
\label{lem:zhounear}
 Suppose that $\Gamma$ is a connected $(G,2)$-arc-transitive graph, where $G \le \Aut(\Gamma)$.  Let $(\alpha, \beta, \gamma)$ be a $2$-arc of $\Gamma$ and define $H:= G_{\alpha \beta \gamma}$.  Then the following are equivalent:
 \begin{itemize}
  \item[(i)] there exist both an integer $c \ge 3$ and a $G$-orbit $\mathcal{C}$ on $c$-cycles of $\Gamma$ such that $\Gamma$ is a near-polygonal graph with set of distinguished cycles $\mathcal{C}$;
  \item[(ii)] $H$ fixes at least one vertex in $\Gamma(\gamma) \backslash \{\beta\}$;
  \item[(iii)] there exists $g \in N_G(H)$ such that $(\alpha, \beta)^g = (\beta, \gamma)$.
 \end{itemize}
\end{lem}

%
%

\section{Constructions of graphs with a permutable matching}
\label{sect:constructions}

In this section, we provide some constructions of graphs with permutable matchings.  We begin with a construction that shows that, for any $m \ge 2$, there are graphs that are neither edge- nor even vertex-transitive that contain a permutable $m$-matching.  

\begin{const}
\label{const:star}
  Let $\Gamma$ be a graph with a vertex $\alpha$ of degree $m$ and a group of automorphisms $G$ such that $G_{\alpha}^{\Gamma(\alpha)}\cong S_m$.  Define a new graph $Q\Gamma$ to be the graph obtained by subdividing every edge of $\Gamma$ into a path of length $2$.
\end{const}

It is not difficult to see that $Q\Gamma$ contains a permutable $m$-matching.  In particular, if $\Gamma = K_{1,m}$, then $\Aut(Q\Gamma) \cong S_m$ and $\Aut(Q\Gamma)$ has three orbits on vertices and two orbits on edges; the orbit of edges that do not all share a common endpoint is a permutable $m$-matching.  

Obviously, it is possible to construct other such examples; we mention another couple here.

\begin{const}
\label{const:subdivnotM}
 Let $\Gamma$ be a graph with a permutable matching $\M$.  Let $Q_{\M}\Gamma$ be the graph obtained by subdividing every edge not in $\M$ into a path of length $2$.  
\end{const}

\begin{const}
 \label{const:subdivMtwice}
 Let $\Gamma$ be a graph with a permutable matching $\M$.  Let $Q^{\M}\Gamma$ be the graph obtained by subdividing every edge in $\M$ into a path of length $3$.
\end{const}

The graphs produced from these constructions may have less symmetry than the original graphs; for instance, these constructions may take vertex-, edge-, or arc-transitive graphs and produce graphs that are not vertex-, edge-, or arc-transitive.  For this reason, we will henceforth restrict ourselves to graphs $\Gamma$ containing a $G$-permutable matching that are also $G$-arc-transitive.  Perhaps the most obvious examples of graphs with permutable $m$-matchings are also examples of vertex-biprimitive graphs with permutable $m$-matchings.

\begin{prop}
 \label{prop:compbip}
 For every $m \ge 2$, the complete bipartite graph $K_{m,m}$ has degree $m$ and there exists $G \le \Aut(K_{m,m})$ such that $K_{m,m}$ is $G$-vertex-biprimitive and $K_{m,m}$ contains a $G$-permutable $m$-matching.
\end{prop}

\begin{proof}
 We take $G = \Aut(K_{m,m}) \cong S_m \Wr S_2$.  The group $G$ preserves the partition of the vertices into two sets of size $m$, and any perfect matching will be a permutable $m$-matching.
\end{proof}

Inspired by the example of complete bipartite graphs, the following construction demonstrates that it is quite easy to construct arc-transitive graphs with permutable matchings for any $m$:

\begin{const}
\label{const:Kmm}
Let $\Gamma$ be an arc-transitive graph with automorphism group $H$ and let $m$ be any fixed natural number.  Define $\Gamma(m)$ as the \emph{lexicographical product} of $\Gamma$ with $\overline{K}_{m}$: that is, $V(\Gamma(m)) = \{ (\eta, i) : \eta \in V(\Gamma), \space 1 \le i \le m \}$, with $(\eta, i)$ adjacent to $(\theta, j)$ if and only if $\eta$ is adjacent to $\theta$ in $\Gamma$.
\end{const}

If $\Gamma(m)$ is constructed from an $H$-arc-transitive graph $\Gamma$ as in Construction \ref{const:Kmm} with $H=\Aut(\Gamma)$, then $S_m \Wr H \le \Aut(\Gamma(m))$. For $G:= S_m \Wr H$, $\Gamma(m)$ is $G$-arc-transitive, and, for any edge $\{\alpha, \beta\}$ in $\Gamma$, the set $\M:= \{ \{(\alpha, i), (\beta, i)\} : 1 \le i \le m \}$ is a $G$-permutable $m$-matching of $\Gamma(m)$. 

Another construction which yields infinitely many such graphs from a $G$-arc-transitive graph $\Gamma$ with a $G$-permutable $m$-matching is the following.

\begin{const}
\label{const:cover}
Let $\Gamma$ be a $G$-arc-transitive graph with a $G$-permutable $m$-matching $\M = \{ (\alpha_i, \beta_i) : 1 \le i \le m \}$. Let $E$ denote the edge set of $\Gamma$ and let $T$ denote the set of edges of a spanning tree of $\Gamma$ that contains each of the edges of $\M$.  Let $\Z_p$ denote the cyclic group of order $p$, where $p$ is a prime. Let $H:= \Z_p^{|E| - |T|}$; $H$ is a $\Z_{p}$-vector space. Let $X$ be a basis for $H$, so $|X|=|E| - |T|$.  Define $\Gamma_p$ to be the derived regular cover of the voltage graph defined by assigning a distinct element of $X$ to each co-tree edge of $\Gamma$.  
\end{const}

By Lemma \ref{lem:allautlifts}, if $\Gamma$ is a $G$-arc-transitive graph with $G$-permutable $m$-matching $\M = \{ \{\alpha_i, \beta_i\} : 1 \le i \le m \}$, then $G$ lifts to a group $\widetilde{G}$ of automorphisms of $\Gamma_p$, and it follows that 
\[\M_p := \{ \{(\alpha_i, 1), (\beta_i, 1) \} : 1 \le i \le m \} \]
is itself a $\widetilde{G}$-permutable $m$-matching of $\Gamma_p$.

What last these two constructions have in common is that the graphs that are produced are not quasiprimitive on vertices: in each case, the full automorphism group of the graph produced contains an intransitive normal subgroup.  Moreover, the groups $G$ chosen above for the graphs arising from Construction \ref{const:Kmm} are always locally imprimitive.  It makes sense, then, to study the $G$-arc-transitive graphs that have $G$-permutable matchings that are $G$-vertex-quasiprimitive or $G$-vertex-biquasiprimitive.  Indeed, such graphs exist.  The \textit{odd graph} $O_n$ has one vertex for each of the $(n-1)$-element subsets of a $(2n-1)$-element set, and vertices are adjacent if and only if the corresponding subsets are disjoint.  As the following result shows, there is at least one vertex-quasiprimitive (and, in fact, vertex-primitive) graph with a permutable $m$-matching for every $m \ge 3$.

\begin{thm}
\label{thm:oddgraphs}
For every $m \ge 3$, the odd graph $O_m$ has degree $m$ and there exists $G \le \Aut(O_m)$ such that $G \cong S_{2m-1}$, $O_m$ is $G$-vertex-primitive, and $O_m$ contains a $G$-permutable $m$-matching.  
\end{thm}

\begin{proof}
 We identify the vertices of $O_m$ with subsets of size $m-1$ of $\{1, 2, \dots, 2m-1\}$. Then $S_{2m-1}$ is primitive on the sets of size $m-1$: the stabilizer of each subset is isomorphic to $S_{m-1} \times S_m$, a maximal subgroup of $S_{2m-1}$ which is core-free (that is, the intersection of all conjugates of the subgroup is trivial; see \cite{Aschbacher}).  Hence there is $G \le \Aut(O_m)$ such that $G \cong S_{2m-1}$ and $O_m$ is $G$-vertex-primitive.  
 
 For each $i$ such that $1 \le i \le m$, define the sets $S_i := \{1, \dots m\} \backslash \{i\}$ and $T_i := \{i\} \cup \{m+1, \dots, 2m-2\}$.  Since vertices of $O_m$ are identified with subsets of $\{1, \dots, 2m-1\}$ of size $m-1$, each $S_i$ and each $T_i$ is a vertex of $O_m$, and, furthermore, \[\M := \{ \{S_i, T_i\} : 1 \le i \le m\}\] is a matching of size $m$.  Let $H:= \Sym(\{1, \dots, m\}) \le G$.  We note that $H \cong S_m$ and $H$ stabilizes $\M$ setwise but allows the edges of $\M$ to be permuted as we please.  Therefore, $\M$ is $H$-permutable, so $\M$ is $G$-permutable, as desired. 
\end{proof}

One might expect that if $\Gamma$ has a group of automorphisms $G$ such that (i) $\Gamma$ has a $G$-permutable $m$-matching, (ii) $G$ has a nontrivial normal subgroup $N$ that is intransitive on vertices, and (iii) $\Gamma$ does not have an induced subgraph isomorphic to $K_{m,m}$ (i.e., if $\Gamma$ does not arise from Construction \ref{const:Kmm}), then the normal quotient graph $\Gamma_N$ should also have a permutable $m$-matching.  However, as the following construction shows, more exotic examples can arise.

\begin{const}
\label{const:nearpolygcover}
Let $\Gamma$ be a $(G,2)$-arc-transitive, near-polygonal graph of degree $m \ge 3$ such that $\Gamma$ does not contain a $G$-permutable $m$-matching and $(\Gamma, G)$ is locally-$S_m$.  Let $E$ denote the edge set of $\Gamma$ and let $T$ denote the set of edges of a spanning tree of $\Gamma$.  Let $\Z_p$ denote the cyclic group of order $p$, where $p$ is a prime. Let $H:= \Z_p^{|E| - |T|}$; $H$ is a $\Z_{p}$-vector space. Let $X$ be a basis for $H$, so $|X|=|E| - |T|$.  Define $\Gamma_p$ to be the derived regular cover of the voltage graph defined by assigning a distinct element of $X$ to each co-tree edge of $\Gamma$.  
\end{const}

\begin{prop}
\label{prop:cover}
The graph $\Gamma_p$ created from Construction \ref{const:nearpolygcover} contains a permutable $m$-matching.
\end{prop}

\begin{proof}
Let $\alpha$ be a vertex of $\Gamma$ with $\Gamma(\alpha) = \{ \beta_1, \dots, \beta_m \}$.  By Lemmas \ref{lem:lift} and \ref{lem:allautlifts}, $G$ lifts to a group of automorphisms $\widetilde{G}$ of $\Gamma_p$ and there is a group homomorphism $\phi: G \rightarrow \Aut(H)$, where the action is induced on a generating set of all closed walks based at the vertex $\alpha$.  Since $\Gamma$ is near-polygonal and $(G,2)$-arc-transitive, each $2$-arc $(\beta_i, \alpha, \beta_j)$ is contained in a unique cycle $C_{i,j}$, and $G_\alpha$ is transitive on these cycles. Define $h_i$ to be the voltage of the walk $W_i$, where $W_i$ is the concatenation of all cycles $C_{i,j}$ such that $j \neq i$.  Note that, since $m \ge 3$, the cycles $C_{i,j}$ are distinct, and $X$ is a basis for $H$, the $h_i$ are all pairwise distinct.  If $g \in G_\alpha$ and $\beta_i^{g} = \beta_j$, then the induced action of $g$ on $H$ sends $h_i$ to $h_j$.  If $\xi_i$ is the voltage of the arc $(\alpha, \beta_i)$, then the matching $\{ \{(\alpha, h_i), (\beta_i, \xi_i + h_i) \} : 1 \le i \le m \}$ is $\widetilde{G}$-permutable.
\end{proof}

\begin{cor}
\label{cor:nearpolygmatchings}
For each $m \ge 3$, there exist infinitely many graphs $\Gamma$ with a group of automorphisms $G$ such that
\begin{itemize}
\item[(i)] $\Gamma$ is $(G,2)$-arc-transitive,
\item[(ii)] $(\Gamma,G)$ is locally-$S_{m}$,
\item[(iii)] $\Gamma$ contains a $G$-permutable $m$-matching, and
\item[(iv)] $G$ has a nontrivial normal subgroup $N$ that has more than two orbits on $V(\Gamma)$,
\end{itemize}
yet $\Gamma_N$ does not contain a $G$-permutable $m$-matching.
\end{cor}

\begin{proof}
For each $m \ge 3$, we can take $\overline{\Gamma}$ to be $K_{m+1}$, the $m$-dimensional hypercube $Q_m$, or the folded $m$-dimensional hypercube, each of which satisfies the hypotheses of Construction \ref{const:nearpolygcover}.  To see that the hypercube $Q_m$ has no $G$-permutable $m$-matching, we first identify the vertices of $Q_m$ with binary $m$-tuples and note that $\Aut(Q_m) \cong S_2 \Wr S_m$. Consequently, $G \le \Aut(Q_m)_\alpha$ for some vertex $\alpha$, which without a loss of generality is the all zeros $m$-tuple.  Hence $G$ must preserve distances of vertices from $\alpha$, and so, in order for $G$ to act like $S_m$ on $m$ distinct $m$-tuples with the same number of $0$'s and $1$'s, there is either exactly one $0$ or exactly one $1$.  Hence, without a loss of generality, the edges in the $G$-permutable $m$-matching are all from vertices at distance $1$ from $\alpha$ to vertices at distance $2$ from $\alpha$.  However, since every $2$-path is contained in a unique $4$-cycle, the action on the edges in the matching cannot be permutable: once an edge in the matching is fixed, necessarily another neighbor of $\alpha$ is fixed, which fixes another edge in the matching, a contradiction to permutability.  The argument for the folded $m$-dimensional hypercube is analogous.   

The result now follows from Proposition \ref{prop:cover}, taking $\Gamma = \overline{\Gamma}_p$, where $p$ ranges over all primes. 
\end{proof}

\section{The local structure of graphs with a permutable matching}
\label{sect:local}

In this section, we prove results about the local structure of a $G$-arc-transitive graph with a $G$-permutable $m$-matching, that is, we prove results about the stabilizer of a vertex and the size of the neighborhood of a vertex in such a graph.  This first result, which has a similar proof to that of \cite[Theorem 1.1]{Song}, provides information about the edge stabilizer of an arc-transitive graph with a permutable $m$-matching when $m$ is large enough.   

\begin{prop}
\label{prop:Am-1Comp}
 Let $\Gamma$ be a $G$-arc-transitive graph with a $G$-permutable $m$-matching, where $m \ge 6$.  If $\{\alpha, \beta\}$ is an edge of $\Gamma$, then there is a subgroup $U \le G_{\alpha\beta}$ such that $U^{\Gamma(\alpha)}$ has a composition factor isomorphic to $A_{m-1}$.
\end{prop}

\begin{proof}
Let $\M$ be a $G$-permutable $m$-matching containing $\{\alpha, \beta\}$, where \[\M = \left\{ e = e_1 = \{\alpha, \beta \}, e_2, \dots, e_m \right\}.\]
Note that $G_{\M}^{\M} \cong S_m$ and $G_{e \M}^{\M \backslash e} \cong S_{m-1}$.  Let $K:= \{g \in G_{e \M} : e_i^g = e_i, 1\le i \le m\}$, the kernel of the action of $G_{e\M}$ on $\M$.  We have $K \lhd G_{e \M}$, $G_{e \M}/K \cong S_{m-1}$, and hence $A_{m-1}$ is a composition factor of $G_{e \M}$ (since $m \ge 6$, $A_{m-1}$ is simple).

Now, consider the subgroup $G_{\alpha \beta}$ of $G_e$.  We have $G_{\alpha \beta} \lhd G_e$ since it has index at most two, and so $\big( G_{\alpha \beta}^{[1]}\big)_\M \lhd G_{\alpha\beta \M} \lhd G_{e \M}$.  Let $P = (\gamma_0 = \alpha, \gamma_1 = \beta, \gamma_2, \dots, \gamma_n)$ be a path in $\Gamma$ such that $G_{\alpha\beta\gamma_2\dots\gamma_n}^{[1]} = 1$.  Hence $$1 = \big(G_{\alpha\beta\gamma_2\dots\gamma_n}^{[1]}\big)_\M \unlhd \dots \unlhd \big(G_{\alpha \beta \gamma_2}^{[1]}\big)_\M \unlhd \big(G_{\alpha\beta}^{[1]}\big)_\M \unlhd G_{\alpha \beta \M}.$$  Since $A_{m-1}$ is a composition factor of $G_{e \M}$ and $G_{\alpha \beta}$ has index at most two in $G_{e}$, $A_{m-1}$ must be a composition factor of $G_{\alpha\beta \M}$.  If $A_{m-1}$ is a composition factor of either $G_{\alpha \beta \M}^{\Gamma(\alpha)}$ or $G_{\alpha \beta \M}^{\Gamma(\beta)}$, then we are done.  Otherwise, $A_{m-1}$ is a composition factor of $\big(G_{\alpha}^{[1]}\big)_\M \cap \big(G_{\beta}^{[1]}\big)_\M= \big(G_{\alpha \beta}^{[1]}\big)_\M$.  Let $\ell$ be the largest integer such that $A_{m-1}$ is a composition factor of $\big(G_{\alpha \beta \gamma_2 \dots \gamma_\ell}^{[1]}\big)_\M$.  This implies that $A_{m-1}$ is not a composition factor of $\big(G_{\alpha \beta \gamma_2 \dots \gamma_\ell\gamma_{\ell+1}}^{[1]}\big)_\M$, and so $A_{m-1}$ must be a composition factor of $\big(G_{\alpha \beta \gamma_2 \dots \gamma_{\ell}}^{[1]}\big)_\M/ \big(G_{\alpha \beta \gamma_2 \dots \gamma_\ell\gamma_{\ell +1}}^{[1]}\big)_\M$.  Since   
\[ \big(G_{\alpha \beta \gamma_2 \dots \gamma_{\ell}}^{[1]}\big)_\M/ \big(G_{\alpha \beta \gamma_2 \dots \gamma_\ell\gamma_{\ell +1}}^{[1]}\big)_\M \cong \Big( \big(G_{\alpha \beta \gamma_2 \dots \gamma_{\ell}}^{[1]}\big)_\M\Big)^{\Gamma(\gamma_{\ell+1})} {\lhd\lhd} \; G_{\gamma_{\ell} \gamma_{\ell+1} \M}^{\Gamma(\gamma_{\ell+1})},\]
$A_{m-1}$ is a composition factor of $G_{\gamma_{\ell}\gamma_{\ell+1} \M}^{\Gamma(\gamma_{\ell+1})}$.  Since $\Gamma$ is $G$-arc-transitive, $G_{\gamma_{\ell}\gamma_{\ell +1} \M} \cong U \le G_{\alpha \beta}$, and hence $A_{m-1}$ is a composition factor of $U^{\Gamma(\alpha)}$ for some $U \le G_{\alpha \beta}$, as desired. 
\end{proof}

A consequence of this result is that the degree of a vertex in an arc-transitive graph with an permutable $m$-matching is at least $m$ when $m \ge 6$; in fact, we can classify the graphs with degree less than $m$ and a permutable $m$-matching. 

\begin{proof}[Proof of Theorem \ref{thm:degreem}]
 By Proposition \ref{prop:Am-1Comp}, when $m\geq 6$, for an edge $\{\alpha, \beta\}$ of $\Gamma$ there exists $U \le G_{\alpha \beta}$ such that $U^{\Gamma(\alpha)}$ has a composition factor isomorphic to $A_{m-1}$.  For $m \ge 5$, the smallest faithful permutation representation of $A_m$ has degree $m$.  Since $U$ fixes $\beta \in \Gamma(\alpha)$, this implies that $|\Gamma(\alpha)| - 1 \ge m - 1$. When $m = 1$ and $m = 2$, the result is clear since $\Gamma$ is connected.  When $m = 3$, since the graph is connected and arc-transitive, the degree of the graph is at least two.  If the degree of $\Gamma$ is exactly two, then $\Gamma$ is a cycle, and the result follows by noting that $\Gamma$ must have at least six vertices and that $\Aut(\Gamma)$, which is a dihedral group, must have order divisible by three. 
 
 We are left with the cases $m = 4$ and $m = 5$.  In either case, if the degree of such a graph $\Gamma$ is $2$, then $\Gamma$ is a cycle, and $\Aut(\Gamma)$ contains no section isomorphic to $S_4$.  If the degree of such a graph $\Gamma$ is $3$, then, by a famous result of Tutte \cite{Tutte}, the order of a vertex stabilizer divides $48$, and hence the order of an edge stabilizer divides $(48 \cdot 2)/3 = 32$.  If $m \ge 4$, $G = \Aut(\Gamma)$, and the permutable matching is $\M$, then $3$ divides the order of the stabilizer of an edge in $G$, since the stabilizer of an edge of $\M$ can permute three other edges of $\M$ in any way. Thus there is no graph of degree $3$ with a permutable $4$-matching.  
 
 Finally, assume $\Gamma$ is regular of degree $4$ and that $\M$ is a $G$-permutable $5$-matching for $G = \Aut(\Gamma)$. Let $\M = \{e_1, e_2, e_3, e_4, e_5\}$, where each $e_i = \{\alpha_i, \beta_i\}$.  Consider a shortest path $P_2$ from a vertex of $e_1$ to a vertex of $e_2$.  Without loss of generality, the path is between $\alpha_1$ and $\alpha_2$.  Since $\M$ is permutable, there are elements $g_i$ in $G_\M$ that fix $e_1$ and map $e_2$ to $e_i$, $3 \le i \le 5$, and so there exist shortest paths from $e_1$ to $e_i$, where $2 \le i \le 5$, that are all of the same length.   
 
 We claim that there must exist a path from $\alpha_1$ to $e_i$ with the same length as $P_2$ for each $i$.  If $\alpha_1^{g_i} = \alpha_1$, then there is a shortest path from $\alpha_1$ to $\alpha_i$, so assume that $\alpha_1^{g_i} = \beta_1$.  Consider a fourth edge $e_j$.  Suppose first that $\alpha_1^{g_j} = \alpha_1$ (so $P_2^{g_j}$ is a shortest path from $e_1$ to $e_j$ starting at $\alpha_1$). Since $\M$ is $G$-permutable, there is $g \in G_\M$ such that $e_1^g = e_1$, $e_2^g = e_2$, and $e_j^g = e_i$.  If $\alpha_1^g = \alpha_1$, then the path $P_2^{g_jg}$ is a shortest path from $e_1$ to $e_i$ starting at $\alpha_1$.  On the other hand, if $\alpha_1^g = \beta_1$, then, since $e_2^g = e_2$, $P_2^g$ is a shortest path from $e_1$ to $e_2$ starting at $\beta_1$, and so there is a shortest path from $e_1$ to $e_2$ starting at each of $\alpha_1$ and $\beta_1$.  Since $\M$ is permutable, this implies that there exists a shortest path from $e_1$ to $e_i$ starting at $\alpha_1$.  Suppose next that $\alpha_1^{g_j} = \beta_1$ (so that $P_2^{g_j}$ is a shortest path from $e_1$ to $e_j$ starting at $\beta_1$).  Since $\M$ is $G$-permutable, there exists $x \in G_\M$ such that $e_1^x = e_1$, $e_2^x = e_i$, and $e_j^x = e_j$.  If $\alpha_1^x = \alpha_1$, then $P_2^x$ is a shortest path from $e_1$ to $e_i$ starting at $\alpha_1$, whereas if $\alpha_1^x = \beta_1$, then $P_2^{g_jx}$ is a shortest path from $e_1$ to $e_j$ starting at $\alpha_1$, so there shortest path from $e_1$ to $e_j$ starting at each of $\alpha_1$ and $\beta_1$.  Since $\M$ is permutable, this implies there is a shortest path from $e_1$ to $e_i$ starting at $\alpha_1$.  Therefore, we can always find a shortest path from $e_1$ to $e_i$ starting at $\alpha_1$ for each $i$, $2 \le i \le 5$.  
 
If necessary, we relabel the vertices in $e_3$, $e_4$, and $e_5$ so that there is a shortest path from $\alpha_1$ to $\alpha_i$ for each $i$, $2 \le i \le 5$, and we denote these paths by $P_i$.  For each $i$, let 
 \[P_i = (\alpha_1 = \gamma_{i,0}, \gamma_{i,1}, \dots, \gamma_{i,n} = \alpha_i).\]
 Since $|\Gamma(\alpha_1) \backslash \{\beta_1\}| = 3$, at least two $P_i$ go through the same neighbor of $\alpha_1$, say $\gamma = \gamma_{2,1} = \gamma_{3,1}$.  Consider $h \in G_\M$ such that $h$ acts on $\M$ as the permutation $(3 \; 4 \; 5)$.  Note that, since the induced action of $h$ on $\M$ has order $3$, if $\alpha_1^h = \beta_1$, then we could choose $h^2$ instead, so we may assume that $\alpha_1^h = \alpha_1$ and $\alpha_2^h = \alpha_2$.  Assume  first that $\gamma^h \neq \gamma$.  This implies that $\Gamma(\alpha_1) = \{\beta_1, \gamma, \gamma^h, \gamma^{h^2}\}$ and that there is a shortest path from $\alpha_1$ to $\alpha_2$ through each of $\gamma, \gamma^h$, and $\gamma^{h^2}$.  However, this implies, by the permutability of $\M$, that there is a shortest path from $\alpha_1$ to each $\alpha_i$ through each of $\gamma$, $\gamma^h$, and $\gamma^{h^2}$.  Thus we may choose each $P_i$ so that $\gamma_{i,1} = \gamma$.  On the other hand, if $\gamma^h = \gamma$, then there is a path from $\alpha_1$ to $\alpha_i$ through $\gamma$ for each $i$, namely, $P_4' := P_3^h$ goes from $\alpha_1$ to $\alpha_4$ and $P_5':= P_3^{h^2}$ goes from $\alpha_1$ to $\alpha_5$.  Hence, in any case we may assume that $\gamma_{i,1} = \gamma$ for all $i$.  However, $\gamma$ has exactly three neighbors that are not $\alpha_1$.  We apply a similar argument for the $\gamma_{i,2}$, $\gamma_{i,3}$, etc., and reach a contradiction: either $\Gamma$ is disconnected or a vertex has degree greater than $4$.  Therefore, there is no connected graph of degree $4$ with a permutable $5$-matching, and the result holds.
 \end{proof}

\section{Locally primitive, arc-transitive graphs with degree $m$ and a permutable $m$-matching}
\label{sect:degreem}

Given that a graph with a permutable $m$-matching has degree at least $m$ when $m \ge 4$ and given the constructions from Section \ref{sect:constructions}, it makes sense to study arc-transitive, locally primitive graphs of degree $m$ that contain a permutable $m$-matching.  The following results show that such graphs $\Gamma$ with a group of automorphisms $G$ do have a nice structure with respect to nontrivial normal subgroups $N$ of $G$ such that $N$ is intransitive on vertices.

\begin{lem}
\label{lem:locallySm}
Let $\Gamma$ be a $G$-arc-transitive graph with degree $m \ge 6$ and let $(\Gamma, G)$ be locally-$S_m$.  Either $\Gamma$ contains a $G$-permutable $m$-matching or $\Gamma$ is near-polygonal.
\end{lem}

\begin{proof}
Let $\Gamma$ be such a graph, and let $\alpha$ be a vertex with $\Gamma(\alpha) =\{\beta_1, \dots, \beta_m\}$.  Since $(\Gamma,G)$ is locally-$S_m$, $\Gamma$ is $G$-locally primitive; in fact, $G_\alpha^{\Gamma(\alpha)} \cong S_m$ and, for each $1 \le i \le m$, $G_{\alpha \beta_i}^{\Gamma(\alpha) \backslash \{\beta_i\}} \cong G_{\alpha \beta_i}^{\Gamma(\beta_i) \backslash \{\alpha\}} \cong S_{m-1}$.  Because $G_{\alpha \beta_1}$ has nontrivial layer (that is, the group generated by its subnormal quasisimple groups is nontrivial; see \cite{Aschbacher}), then, by \cite[Theorem 2.12]{VanBon}, $G_{\alpha \beta_1}^{[1]} = 1$.  Thus 
\[[G_\alpha^{[1]}, G_{\beta_i}^{[1]}] \le G_\alpha^{[1]} \cap G_{\beta_i}^{[1]} = G_{\alpha \beta_i}^{[1]} = 1,\]
 i.e., for each $i$, the elements of $G_{\alpha}^{[1]}$ and $G_{\beta_i}^{[1]}$ commute.  Since $G_\alpha^{[1]} \lhd G_{\alpha \beta_i}$, 
\[ G_\alpha^{[1]} \cong G_\alpha^{[1]}/(G_\alpha^{[1]} \cap G_{\beta_i}^{[1]}) \cong  \left(G_{\alpha}^{[1]}\right)^{\Gamma(\beta_i) \backslash \{\alpha\}} \lhd G_{\alpha \beta_i}^{\Gamma(\beta_i) \backslash \{\alpha\}}.\]
Since the only normal subgroups of $S_{m-1}$ when $m-1 \ge 5$ are $1$, $A_{m-1}$, and $S_{m-1}$, we conclude that $G_\alpha^{[1]}$ is isomorphic to one of $1$, $A_{m-1}$, or $S_{m-1}$.

\subsection*{5.1.1: The case where $G_{\alpha}^{[1]} \cong S_{m-1}$}\
\label{5.1S_m-1}

Suppose first that $G_{\alpha}^{[1]} \cong S_{m-1}$.  Define \[L_i := G_\alpha^{[1]}G_{\beta_i}^{[1]} \cong G_\alpha^{[1]} \times G_{\beta_i}^{[1]} \cong S_{m-1} \times S_{m-1}.\] We note that $L_i \le G_{\alpha \beta_i}$.  Since $G_{\alpha \beta_i}^{[1]} = 1$, we have
\[G_{\alpha \beta_i}/G_{\alpha}^{[1]} \cong G_{\alpha \beta_i}^{\Gamma(\alpha)} \cong S_{m-1},\] 
and so $|L_i| = |G_{\alpha \beta_i}|$ and hence $G_{\alpha \beta_i} = L_i$.  Moreover, when $m -1\ge 5$, $Z(S_{m-1}) = 1$; hence $\big(G_{\beta_i}^{[1]}\big)_{\beta_j}^{\Gamma(\beta_j) \backslash \{\alpha\}} = 1$.  In other words, for any $i \neq j$ we have \[\big(G_{\beta_i}^{[1]}\big)_{\beta_j} = \big(G_{\beta_j}^{[1]}\big)_{\beta_i} = G_{\beta_i}^{[1]} \cap G_{\beta_j}^{[1]}.\]  For each $j \ge 2$, we have 
\[G^{[1]}_{\beta_2 \dots \beta_{j-1}\beta_{j+1} \dots \beta_m} \cong S_2, \] and so we let $G^{[1]}_{\beta_2 \dots \beta_{j-1}\beta_{j+1} \dots \beta_m} = \langle g_{1,j}\rangle.$  If $\Gamma(\alpha) \cap \Gamma(\beta_1) \neq \varnothing$, since $(\Gamma, G)$ is locally-$S_m$, then $\Gamma \cong K_{m+1}$ and $G_\alpha^{[1]} = 1$, a contradiction.  Thus we may pick $\gamma_1 \in \Gamma(\beta_1) \backslash \{\alpha\}$, and we define $\gamma_i := \gamma_1^{g_{1,i}}.$  If $H = \langle g_{1,i} : 2 \le i \le m\rangle$, $H$ stabilizes $\M = \{ \{\beta_i, \gamma_i\} : 1 \le i \le m\}$ setwise, and $H^{\M}\cong S_{m}$, so $\M$ is an $H$-permutable $m$-matching, as desired.

\subsection*{5.1.2: The case where $G_{\alpha}^{[1]} \cong A_{m-1}$}\
\label{5.1A_m-1}

Suppose next that $G_{\alpha}^{[1]} \cong A_{m-1}$.  We define \[L_i := G_\alpha^{[1]}G_{\beta_i}^{[1]} \cong G_\alpha^{[1]} \times G_{\beta_i}^{[1]}\] as above, only now $L_i \cong A_{m-1} \times A_{m-1}$.  Since \[G_{\alpha \beta_i}^{\Gamma(\alpha)} \cong G_{\alpha \beta_i}/G_{\alpha}^{[1]} \cong S_{m-1},\] we have that $|G_{\alpha\beta_i}:L_i| = 2$.  As in the last case,
\[\big(G_{\beta_i}^{[1]}\big)_{\beta_j} = \big(G_{\beta_j}^{[1]}\big)_{\beta_i} = G_{\beta_i \beta_j}^{[1]} \cong A_{m-2}.\]
However, in this case, \[G_{\beta_3 \dots \beta_m}^{[1]} \cong A_2 = 1,\]
and so \[G_{\alpha \beta_3 \dots \beta_m} = G_{\beta_3 \dots \beta_m} \cong G_{\beta_3 \dots \beta_m}/G_{\beta_3 \dots \beta_m}^{[1]}\]
and \[G_{\alpha \beta_3 \dots \beta_m}^{\{\beta_1, \beta_2\}} \cong G_{\alpha \beta_3 \dots \beta_m}/G_{\alpha}^{[1]} \cong S_2.\]  Hence we choose $g \in G_{\alpha \beta_3 \dots \beta_m}$ such that $\beta_1^g = \beta_2$ and $\beta_2^g = \beta_1$.  We may also assume that $\gamma_3^g = \gamma_3$ for some $\gamma_3 \in \Gamma(\beta_3) \backslash \{\alpha\}$; otherwise, we replace $g$ by $gx$, where $x \in G_\alpha^{[1]}$; indeed, this in fact shows that we may assume that $g$ acts as a transposition on $\Gamma(\beta_3) \backslash \{\alpha\}$.  We also remark that $G_{\alpha \beta_i} = \langle L_i, g \rangle$ for $i \ge 3$.

Define \[H:= \langle G_{\beta_i}^{[1]} : 1 \le i \le m \rangle.\]  It is clear that $H \lhd G_\alpha$, and, since $G_\alpha^{[1]} \cap G_{\beta_i}^{[1]} = 1$ and $\big(G_{\beta_i}^{[1]}\big)_{\beta_{j}}=\big(G_{\beta_j}^{[1]}\big)_{\beta_{i}}=G_{\beta_i}^{[1]}\cap G_{\beta_j}^{[1]}$ for each $i$ and $j$, we have 
 \[H \cong H/(G_\alpha^{[1]} \cap H) \cong HG_{\alpha}^{[1]}/G_\alpha^{[1]} \lesssim G_{\alpha}^{\Gamma(\alpha)}.\]  
Moreover, $H \lhd G_\alpha$, so $H$ 
is isomorphic to a normal subgroup of $G_{\alpha}^{\Gamma(\alpha)}$.  Since $H$ has a nontrivial action on $\Gamma(\alpha)$, either $H \cong A_m$ or $H \cong S_m$.  

As in the previous case, $\Gamma(\alpha) \cap \Gamma(\beta_1) = \varnothing$.  We claim now that $H$ has $m-1$ orbits of size $m$ on \[D_2(\alpha) := \bigcup_{i=1}^m\limits \Gamma(\beta_i) \backslash \{\alpha\}.\] Suppose first that $x_i, y_i \in G_{\beta_i}^{[1]}$ and $\beta_k^{x_i} = \beta_k^{y_i}$.  This means $x_iy_i^{-1} \in \big(G_{\beta_i}^{[1]}\big)_{\beta_k} \le G_{\beta_k}^{[1]}$. Hence, if $\gamma \in \Gamma(\beta_k)$, then $\gamma^{x_iy_i^{-1}} = \gamma$ and $\gamma^{x_i} = \gamma^{y_i}$.  Now suppose $x_i \in G_{\beta_i}^{[1]}$, $x_j \in G_{\beta_j}^{[1]}$, and $\beta_k^{x_i} = \beta_k^{x_j} = \beta_\ell$.  There exists some $r \in \{1, \dots, m\} \backslash \{i,j,k,\ell\}$ and there exist $y_i \in G_{\beta_i}^{[1]}$ and $y_j \in G_{\beta_j}^{[1]}$ such that 
\[ \beta_k^{y_i} = \beta_k^{x_i} = \beta_k^{y_j} = \beta_k^{x_j} = \beta_\ell \text{ and } \beta_r^{y_i} = \beta_r^{y_j} = \beta_r.\]
Thus $y_i, y_j \in G_{\beta_r}^{[1]}$, and, if $\gamma \in \Gamma(\beta_k)$, then $\gamma^{y_i} = \gamma^{y_j}$ from what we just proved above.  Thus \[\gamma^{x_i} = \gamma^{y_i} = \gamma^{y_j} = \gamma^{x_j},\]
and so $H$ has exactly $m-1$ orbits of size $m$ on $D_2(\alpha).$ 

Now, define $X:= \langle H,g \rangle.$  Since $X$ is a $2$-transitive group on $\Gamma(\alpha)$ that contains a transposition, $X^{\Gamma(\alpha)} \cong S_m$. Now, $X \le G_\alpha$ and $H \lhd G_\alpha$, so $H$ is a normal subgroup of $X$.  This implies that the orbits of $H$ on  $D_2(\alpha)$ are an $X$-invariant partition, which we use to find our matching: indeed, suppose $\gamma^H$ is such an orbit.  Then 
\[(\gamma^H)^g = (\gamma^g)^H.\]  
Select the orbit $\gamma_3^H$, where $\gamma_3 \in \Gamma(\beta_3)$ and $\gamma_3^g = \gamma_3$ as above.  Define $\gamma_{i}:=\gamma_{3}^{H}\cap\Gamma(\beta_{i})$.  This implies that $\gamma_3^X = \gamma_3^H$, and hence $X$ stabilizes $\M = \{ \{\beta_i, \gamma_i\} : 1 \le i \le m\}$ setwise and $\M$ is an $X$-permutable $m$-matching, as desired. 

\subsection*{5.1.3: The case where $G_{\alpha}^{[1]} \cong 1$}\
\label{5.1trivial}

The final case is when $G_\alpha^{[1]} = 1$.  This implies that $G_\alpha \cong S_{m}$ and $G_{\alpha \beta_1} \cong S_{m-1}$ with a faithful action on each of $\Gamma(\alpha) \backslash \{\beta_1\}$ and $\Gamma(\beta_1) \backslash \{\alpha\}$.  Hence $G_{\alpha \beta_1 \beta_2}$ fixes a vertex in $\Gamma(\beta_1) \backslash \{\alpha\}$.  Moreover, since $\Gamma$ has degree $m$ and $(\Gamma,G)$ is locally-$S_m$, $\Gamma$ is a $(G,2)$-arc-transitive graph.  By Lemma \ref{lem:zhounear}, $\Gamma$ is near-polygonal, as desired.
\end{proof}

We can now prove Theorem \ref{thm:mvalent}, which essentially characterizes arc-transitive, $G$-locally primitive graphs of degree $m$ with a permutable $m$-matching.

\begin{proof}[Proof of Theorem \ref{thm:mvalent}]
Suppose that $\Gamma$ is a $G$-arc-transitive, $G$-locally primitive graph with degree $m \ge 6$ that contains a $G$-permutable $m$-matching $\M$ such that $G$ contains an intransitive normal subgroup $N$ that has more than two orbits of vertices.  Let $\{\alpha, \beta\}$ be an edge of $\M$, let $A$ be the $N$-orbit containing $\alpha$, and let $B_1$ be the $N$-orbit containing $\beta$.  Since $G$ is edge-transitive and the $N$-orbits of vertices are $G$-invariant, all edges of $\Gamma$ are between $N$-orbits; that is, if $\{\gamma, \delta\} \in E(\Gamma)$, then $\gamma, \delta$ are in different $N$-orbits.  Thus $A \neq B_1$.  Up to relabeling, there are three possibilities for $\{\gamma, \delta\}$, where $\{\gamma, \delta\}$ is another edge of $\M$:
\begin{itemize}
	\item[(i)] Neither $\gamma$ nor $\delta$ is in either $A$ or $B_1$.
	\item[(ii)] $\gamma \in A$, $\delta \not\in B_1$.
	\item[(iii)] $\gamma \in A$, $\delta \in B_1$.
\end{itemize}

\subsection*{5.2.1: Neither $\gamma$ nor $\delta$ is in either $A$ or $B_1$}\
\label{5.2.1}

Since $\{\alpha, \beta\}, \{\gamma, \delta\} \in \M$, the four vertices $\alpha$, $\beta$, $\gamma$, $\delta$ are in distinct $N$-orbits, and, since $\M$ is a $G$-permutable $m$-matching, no two vertices of $V(\M)$ are in the same $N$-orbit.  We may thus view the action of $G_\M$ on $V(\M)$ as an action on the $N$-orbits containing the vertices. This means there will be a $G/N$-permutable $m$-matching in the quotient graph $\Gamma_N$, and we are done.

\subsection*{5.2.2: $\gamma \in A$, $\delta \not\in B_1$}\
\label{5.2.2}

Here, $\M$ is of the form $$\left\{ \{\alpha_1, \beta_1\} = \{\alpha, \beta\}, \{\alpha_2, \beta_2\}, \dots, \{\alpha_m, \beta_m\} \right\},$$ where $\alpha_i \in A$ and $\beta_i \in B_i$ for each $i$.  Since $\M$ is permutable, this implies that $B_1$, \dots, $B_m$ are distinct $N$-orbits.  Moreover, if $\mathcal{B} = \{\beta = \beta_1, \beta_2, \dots, \beta_m\}$, since each edge contains a vertex in the $N$-orbit $A$ and $\M$ is permutable, the stabilizer of $A$ in $G$ acts as the full symmetric group on $\mathcal{B}$, that is, $G_{A\mathcal{B}}^{\mathcal{B}} \cong S_m$.  Moreover, since $\Gamma$ is $G$-locally primitive of degree $m$, $(\Gamma_N,G/N)$ is locally-$S_m$, the neighbors of the vertex $A$ of $\Gamma_N$ are precisely $B_1, \dots, B_m$, and the vertex $\alpha$ has a unique neighbor in each of $B_1, \dots, B_m$.  By Lemma \ref{lem:locallySm}, the result follows.

\subsection*{5.2.3: $\gamma \in A$, $\delta \in B_1$}\
\label{5.2.3}

Now, $\M$ is entirely contained within the $N$-orbits $A$ and $B_1$, i.e. $\M$ is of the form 
\[\left\{ \{\alpha_1, \beta_1\} = \{\alpha, \beta\}, \{\alpha_2, \beta_2\}, \dots, \{\alpha_m, \beta_m\} \right\},\]
where $\alpha_i \in A$ and $\beta_i \in B_1$ for all $i$.   Moreover, since $\Gamma$ is $G$-locally primitive, the induced subgraph $\Gamma[V(A) \cup V(B_1)]$ is a matching.

Since $\Gamma$ is connected, we may select a path $P_2$ from $\alpha_1$ to $\alpha_2$, say \[P_2 = (\alpha_1, \gamma_{1,1}, \gamma_{1,2}, \dots, \gamma_{1,k}, \alpha_2).\]  For each $i$, let the vertex $\gamma_{1,i}$ lie in the $N$-orbit $C_{1,i}$.  For each $i$, consider the orbit $C_{1,i}^{G_{\M}}$, and let $C_{1,0}:= A$.  If, for any $i$, $|C_{1,i}^{G_\M}| > 1$, then, if $\ell$ is the least such $i$, $|C_{1,\ell}^{G_\M}| = m$ (since $\M$ is a $G$-permutable $m$-matching) and $|C_{1,\ell -1}^{G_\M}| = 1$.  This implies further that $(\Gamma_N,G/N)$ is locally-$S_m$, and the result follows by Lemma \ref{lem:locallySm}.  Finally, if $|C_{1,i}^{G_\M}| = 1$ for all $i$, then $\Gamma$ cannot be connected without some vertex having more than one neighbor in an $N$-orbit, a contradiction to the $G$-local primitivity of $\Gamma$.

Therefore, in any case either $\Gamma_N$ contains a $G/N$-permutable $m$-matching or $\Gamma_N$ is a near-polygonal graph with $\left(G/N\right)_A \cong S_m$.
\end{proof}

As discussed after the statement of Theorem \ref{thm:mvalent}, this provides a characterization of graphs with degree $m$ containing a permutable $m$-matching, in the sense that under these conditions, one can keep taking normal quotients of this graph until reaching either a graph with a permutable $m$-matching or a near-polygonal graph where the stabilizer of a vertex acts on its $m$ neighbors like $S_{m}$. Moreover, Theorem \ref{thm:mvalent} is a best-possible characterization in the sense that graphs in each case do exist.  When combined with Construction \ref{const:cover}, Theorem \ref{thm:oddgraphs} shows that for any $m \ge 6$ there exists a connected $G$-arc-transitive, $G$-locally-primitive graph with a $G$-permutable $m$-matching such that $G$ has an intransitive normal subgroup $N$ with more than two orbits of vertices such that $\Gamma_N$ is $G/N$-vertex-quasiprimitive and contains a $G/N$-permutable $m$-matching.  When combined with Construction \ref{const:cover}, Proposition \ref{prop:compbip} shows that for any $m \ge 6$ there exists a connected $G$-arc-transitive, $G$-locally-primitive graph with a $G$-permutable $m$-matching such that $G$ has an intransitive normal subgroup $N$ that has more than two orbits on vertices such that $\Gamma_N$ is $G/N$-vertex-biquasiprimitive and contains a $G/N$-permutable $m$-matching.  Finally, Corollary \ref{cor:nearpolygmatchings} shows that for any $m \ge 6$ there exists a connected $G$-arc-transitive, $G$-locally-primitive graph with a $G$-permutable $m$-matching such that $G$ contains an intransitive normal subgroup $N$ that has more than two orbits on vertices, where $\Gamma_N$ is near polygonal, $(\Gamma_N, G/N)$ is locally-$S_m$, but $\Gamma_N$ does not contain a permutable $m$-matching.


\section{A classification of graphs with a 2-transitive perfect matching}
\label{sect:classification}

This section is devoted to the proof of Theorem \ref{thm:2transperfect}, which classifies the connected graphs that contain a $2$-transitive perfect matching of size $m$. Throughout this section, we will use the following notation.  We define $\Gamma$ to be a graph with a perfect matching $\M$ of $m$ edges such that $\Aut(\Gamma)$ is $2$-transitive on $\M$.  This implies that $|V(\Gamma)| = 2m$ and $V(\Gamma) = V(\M)$.  We write $\M$ as follows:
\[\M=\left\{e_{i}=\{\alpha_{i},\beta_{i}\}\right\}_{i=1}^{m}.\]
We also define $M \le \Aut(\Gamma)$ to be the subgroup of $\Aut(\Gamma)$ preserving $\M$ setwise, i.e., $M := \Aut(\Gamma)_\M$.

We begin with the following observation, which allows us to subdivide the problem into cases.

\begin{lem}
\label{lem:induced}
For any $i,j$ such that $1 \le i < j \le m$, $\Gamma[\alpha_i, \beta_i, \alpha_j, \beta_j] \cong \Gamma[\alpha_1, \beta_1, \alpha_2, \beta_2]$. 
\end{lem}

\begin{proof}
This follows immediately from the $2$-transitivity of $\Aut(\Gamma)$ on $\M$. 
\end{proof}

\begin{lem}
\label{lem:atleastone}
Let $\{\alpha_i, \beta_i\} \in \M$.  Each other edge of $\M$ has at least one endpoint adjacent to either $\alpha_i$ or $\beta_i$.
\end{lem}

\begin{proof}
 Assume that $\M$ contains more than one edge, and let $e_i = \{\alpha_i, \beta_i\} \in \M$.  Since $\Gamma$ is connected, either $\alpha_i$ or $\beta_i$ has another neighbor, say $\gamma$.  Since $\M$ is a perfect matching, $\gamma$ is $\alpha_j$ or $\beta_j$ for some $j$.  Since there is at least one edge from an endpoint of $e_i$ to an endpoint of $e_j$, the result follows by Lemma \ref{lem:induced}.
\end{proof}

We now subdivide the problem based on the induced subgraph $\Gamma[\alpha_1,\beta_1,\alpha_2, \beta_2]$.  

\begin{lem}
 \label{lem:div}
If $\Gamma$ has a matching $\M$ such that $\Aut(\Gamma)_\M$ is $2$-transitive on the edges of $\M$, then the induced subgraph $\Gamma[\alpha_1,\beta_1,\alpha_2, \beta_2]$ will be isomorphic to one of $K_4$, $C_4$, $K_4 \backslash \{e\}$, $P_4$, or a triangle with a pendant edge. 
\end{lem}

\begin{proof}
 This follows from Lemmas \ref{lem:induced} and \ref{lem:atleastone} and exhausting the graphs on four vertices.  See Figure \ref{fig:1} for these induced subgraphs.
\end{proof}

\begin{figure}[h!]
\begin{multicols}{3}
\begin{center}
\begin{tikzpicture}[scale=0.3]
\node[fill, shape=circle] (a1) at (0,8) {};
\node[fill, shape=circle] (b1) at (0,0) {};
\node[fill, shape=circle] (a2) at (5,8) {};
\node[fill, shape=circle] (b2) at (5,0) {};

\draw[line width=2pt] (a1)--(b1) node[pos=.5, left] {$e_{1}$};
\draw[line width=2pt] (a2)--(b2) node[pos=.5, right] {$e_{2}$};
\draw[line width=2pt] (a1)--(a2) ;
\draw[line width=2pt] (a1)--(b2) ;
\draw[line width=2pt] (b1)--(a2) ;
\draw[line width=2pt] (b1)--(b2) ;
\end{tikzpicture}
\end{center}
\begin{center}
\begin{tikzpicture}[scale=0.3]
\node[fill, shape=circle] (a1) at (0,8) {};
\node[fill, shape=circle] (b1) at (0,0) {};
\node[fill, shape=circle] (a2) at (5,8) {};
\node[fill, shape=circle] (b2) at (5,0) {};

\draw[line width=2pt] (a1)--(b1) node[pos=.5, left] {$e_{1}$};
\draw[line width=2pt] (a2)--(b2) node[pos=.5, right] {$e_{2}$};
\draw[line width=2pt] (a1)--(a2) ;
\draw[line width=2pt] (b1)--(b2) ;
\end{tikzpicture}
\end{center}
%

\begin{center}
\begin{tikzpicture}[scale=0.3]
\node[fill, shape=circle] (a1) at (0,8) {};
\node[fill, shape=circle] (b1) at (0,0) {};
\node[fill, shape=circle] (a2) at (5,8) {};
\node[fill, shape=circle] (b2) at (5,0) {};

\draw[line width=2pt] (a1)--(b1) node[pos=.5, left] {$e_{1}$};
\draw[line width=2pt] (a2)--(b2) node[pos=.5, right] {$e_{2}$};
\draw[line width=2pt] (a1)--(a2) ;
\draw[line width=2pt] (a1)--(b2) ;
\draw[line width=2pt] (b1)--(b2) ;
\end{tikzpicture}
\end{center}

\end{multicols}

\begin{multicols}{2}
\begin{center}
\begin{tikzpicture}[scale=0.3]
\node[fill, shape=circle] (a1) at (0,8) {};
\node[fill, shape=circle] (b1) at (0,0) {};
\node[fill, shape=circle] (a2) at (5,8) {};
\node[fill, shape=circle] (b2) at (5,0) {};

\draw[line width=2pt] (a1)--(b1) node[pos=.5, left] {$e_{1}$};
\draw[line width=2pt] (a2)--(b2) node[pos=.5, right] {$e_{2}$};
\draw[line width=2pt] (a1)--(a2) ;
\draw[line width=2pt] (a1)--(b2) ;
\end{tikzpicture}
\end{center}
\begin{center}
\begin{tikzpicture}[scale=0.3]
\node[fill, shape=circle] (a1) at (0,8) {};
\node[fill, shape=circle] (b1) at (0,0) {};
\node[fill, shape=circle] (a2) at (5,8) {};
\node[fill, shape=circle] (b2) at (5,0) {};

\draw[line width=2pt] (a1)--(b1) node[pos=.5, left] {$e_{1}$};
\draw[line width=2pt] (a2)--(b2) node[pos=.5, right] {$e_{2}$};
\draw[line width=2pt] (a1)--(b2) ;
\end{tikzpicture}
\end{center}
\end{multicols}
\caption{The possibilities for $\Gamma[\alpha_1,\beta_1, \alpha_2, \beta_2]$}
\label{fig:1}
\end{figure}
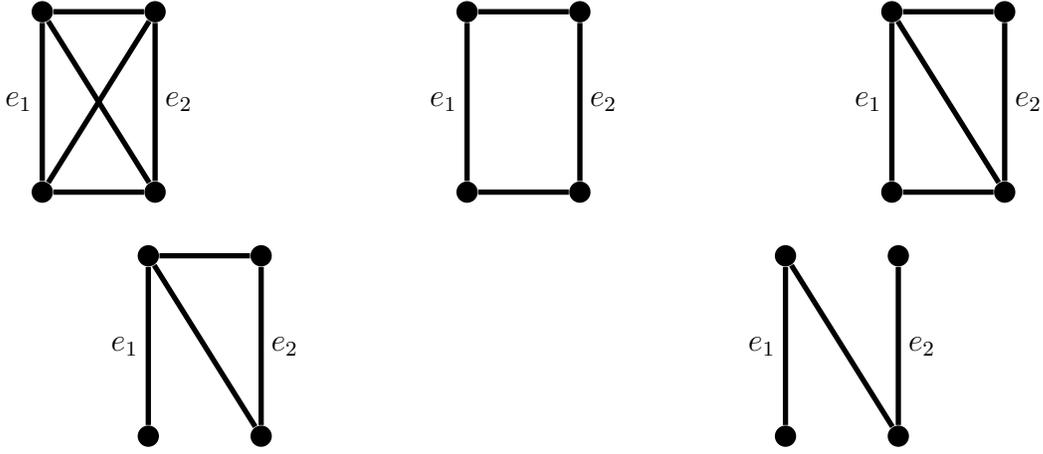

We consider these cases one by one.

\begin{lem}\label{lem:22case}
If $\Gamma[\alpha_{1},\beta_{1},\alpha_{2},\beta_{2}]\cong K_{4}$, then $\Gamma\cong K_{2m}$.
\end{lem}
\begin{proof}
Consider any two vertices $\gamma, \delta \in V(\Gamma)$. In $\M$,  either $\gamma$ and $\delta$ are matched or not. If they are matched, they are the endpoints of some $e_{i}$. If not, one is an endpoint of some $e_{i}$ and the other is an endpoint of some $e_{j}$. But by Lemma \ref{lem:induced}, $\gamma$ and $\delta$ are adjacent in this case as well. Then every pair of vertices is adjacent and $\Gamma\cong K_{2m}$.
\end{proof}

\begin{lem}\label{lem:20case}
There is no graph $\Gamma$ such that $\Gamma[\alpha_{1},\beta_{1},\alpha_{2},\beta_{2}]$ is a triangle with a pendant edge.
\end{lem}
\begin{proof}
Without loss of generality, we let $\alpha_1$ be the vertex with degree $3$ and $\beta_1$ be the vertex with degree $1$.  By the $2$-transitivity of $\Aut(\Gamma)$ on $\M$, there is a $g\in \Aut(\Gamma)$ such that $\{\alpha_{1},\beta_{1}\}^{g}=\{\alpha_{2},\beta_{2}\}$ and $\{\alpha_{2},\beta_{2}\}^{g}=\{\alpha_{1},\beta_{1}\}$. Without loss of generality, $\beta_{2}^{g}=\beta_{1}$ and $\alpha_{2}^{g}=\alpha_{1}$. But because $\alpha_{1}\sim\beta_{2}$, $\alpha_{1}^{g}\sim\beta_{2}^{g}$, we have $\alpha_{1}^{g}\sim\beta_{1}$. But $\alpha_{1}^{g}\in\{\alpha_{2},\beta_{2}\}$, so we have a contradiction.
\end{proof}

\subsection{The case $\Gamma[\alpha_1, \beta_1, \alpha_2, \beta_2] \cong C_4$}

In order to characterize the graphs when the induced subgraph $\Gamma[\alpha_1, \beta_1, \alpha_2, \beta_2]$ is isomorphic to $C_4$, we first need some preliminary results.  A vertex $\gamma \in V(\Gamma)$ is contained in a unique edge $e_i$ of $\M$, so we define $\gamma^c := V(e_i) \backslash \{\gamma\}$, i.e., $\gamma^c$ is the unique vertex adjacent to $\gamma$ in the matching $\M$.  

\begin{lem}
 \label{lem:flip}
 Assume $\Gamma[\alpha_1, \beta_1, \alpha_2, \beta_2] \cong C_4$.  Let $x \in \Sym(V(\Gamma))$ be the permutation of the vertices of $\Gamma$ defined by $\gamma^x = \gamma^c$ for all $\gamma \in V(\M)=V(\Gamma)$, that is, $\alpha_i^x = \beta_i$ and $\beta_i^x = \alpha_i$ for all $i$.  Then $x \in Z(M)$.
\end{lem}

\begin{proof}
 We first need to show that $x \in M$; that is, we need to show that $x \in \Aut(\Gamma)$ and $x$ preserves the matching $\M$ setwise.  Suppose $\gamma, \delta \in V(\Gamma)$ and $\gamma \sim \delta$.  If $\delta = \gamma^c$, then $\gamma^x = \delta$ and $\delta^x = \gamma$, and so $\gamma^x \sim \delta^x$.  If $\delta \neq \gamma^c$, then, since $\Gamma[\alpha_1, \beta_1, \alpha_2, \beta_2] \cong C_4$ and $\gamma \sim \delta$, we have that $\gamma^c \sim \delta^c$, and hence $\gamma^x \sim \delta^x$.  Since $x$ is a permutation of the vertices of a finite graph $\Gamma$ mapping edges to edges, $x \in \Aut(\Gamma)$.  Since $x$ fixes each edge $e_i$, $x \in M$.
 
 We will now show that $x \in Z(M)$.  Let $g \in M$.  For any $\gamma \in V(\Gamma)$, we have:
\begin{align*}
\gamma^{gxg^{-1}} & = ((\gamma^{g})^{x})^{g^{-1}}\\
 & = ((\gamma^{g})^{c})^{g^{-1}}\\
 & = (\gamma^{c})^{gg^{-1}}\\
 & = \gamma^{c}\\ 
 & = \gamma^{x}.
\end{align*}
Therefore, $gxg^{-1} = x$ for all $g \in M$, and so $x \in Z(M)$. 
\end{proof}

\begin{lem}
\label{lem:vertstab2trans}
 Assume $\Gamma[\alpha_1, \beta_1, \alpha_2, \beta_2] \cong C_4$.  For $\gamma \in V(\Gamma)$, if $e$ is the edge of $\M$ containing $\gamma$, then $M_\gamma$ is transitive on $\M \backslash \{ e \}$.  
\end{lem}

\begin{proof}
 Let $e \in \M$ and $e = \{\gamma, \delta\}$.  Since $M$ is $2$-transitive on $\M$, $M_e$ is transitive on $\M \backslash \{e\}$.  Let $e_i, e_j \in \M \backslash \{e\}$.  Then there exists $g \in M_e$ such that $e_i^g = e_j$.  If $g \not\in M_\gamma$, then $gx \in M_\gamma$ and $e_i^{gx} = e_j$, where $x$ is as in Lemma \ref{lem:flip}.  The result follows.
\end{proof}

\begin{lem}
 \label{lem:flipall}
 Assume $\Gamma[\alpha_1, \beta_1, \alpha_2, \beta_2] \cong C_4$.  Define
 \[A_i := \{\alpha_i\} \cup \{\gamma\in V(\Gamma) : i\neq j, \space \gamma\nsim\alpha_{i}\} = \{ \gamma \in V(\Gamma): \gamma \sim \beta_i\}\] and $B_i := \{\gamma\in V(\Gamma)| \gamma\sim\alpha_{i}\}.$  If $g \in \Aut(\Gamma)$ and $e_i^g = e_i$, then $g$ preserves the partition of $V(\Gamma)$ into $A_i \cup B_i$.  Moreover, if $\alpha_i^g = \alpha_i$, then $A_i^g = A_i$ and $B_i^g = B_i$; if $\alpha_i^g = \beta_i$, then $A_i^g = B_i$ and $B_i^g = A_i$.
\end{lem}

\begin{proof}
Since \[A_i = \{\gamma \in V(\Gamma) : \gamma \nsim \alpha_i\} = \{\gamma \in V(\Gamma) : \gamma \sim \beta_i\}\] and
 \[B_i = \{\gamma \in V(\Gamma) : \gamma \nsim \beta_i\} = \{\gamma \in V(\Gamma) : \gamma \sim \alpha_i\},\]
we have that $A_i \cup B_i$ is a partition of $V(\Gamma)$.  Since automorphisms preserve adjacency and nonadjacency, the result follows. 
\end{proof}

\begin{lem}
 \label{lem:partition}
 Assume $\Gamma[\alpha_1, \beta_1, \alpha_2, \beta_2] \cong C_4$.  The vertices of $\Gamma$ can be partitioned into two sets, $A$ and $B$, such that $|A| = |B| = m$, each of $A$ and $B$ contains exactly one endpoint from each edge of $\M$, and either $\Gamma[A] \cong \Gamma[B] \cong K_m$ or $\Gamma[A] \cong \Gamma[B] \cong \overline{K}_m$.
\end{lem}

\begin{proof}
 By Lemmas \ref{lem:vertstab2trans} and \ref{lem:flipall}, for any vertex $\gamma \in V(\Gamma)$, $M_\gamma$ has four orbits on vertices: $\{\gamma\}$, $\{\gamma^c\}$, $\Gamma(\gamma) \backslash \{\gamma^c\}$, and $\Gamma(\gamma^c) \backslash \{\gamma\}$.  Without loss of generality, we may let $\{\gamma, \gamma^c\} = e_1$, $\Gamma(\gamma^c) \backslash \{\gamma\} = \{\alpha_i : i \ge 2\}$, and $\Gamma(\gamma) \backslash \{\gamma^c\} = \{\beta_i : i \ge 2\}$.  Moreover, by Lemma \ref{lem:vertstab2trans}, there exists $h \in M_{\alpha_2}$ such that $\alpha_2^h = \alpha_2$ and $e_1^h = e_3$.  
 
 Suppose first that $\gamma^h = \alpha_3$.  Let $\gamma = \alpha_1$, and let $A := \{\alpha_i : 1 \le i \le m \}$ and $B := \{\beta_i : 1 \le i \le m\}$.  By Lemma \ref{lem:vertstab2trans}, for each $i \ge 2$ there exists $g_i \in M_{\alpha_1}$ such that $\alpha_1^{g_i} = \alpha_1$ and $\alpha_2^{g_i} = \alpha_i$.  Note that \[A_2 = (A_2 \cap A_1) \cup (A_2 \cap B_1),\]  
 \[B_2 = (B_2 \cap A_1) \cup (B_2 \cap B_1),\]
 where $A_i$ and $B_i$ are defined as in the statement of Lemma \ref{lem:flipall}.  Since $\alpha_2^h = \alpha_2$, $A_2^h = A_2$ and $B_2^h = B_2$, and so either (i) $A^h = A$ and $B^h = B$ or (ii) $h$ swaps $(A_1 \cap A_2)$ and $(B_1 \cap A_2)$ and $h$ swaps $(A_1 \cap B_2)$ and $(B_1 \cap B_2)$.  However, $\alpha_2 \in A_1 \cap A_2$, so we have $A^h = A$ and $B^h = B$.  Let $\alpha_i, \alpha_j \in A$, $i \neq j$.  Since $A$ is invariant under $M_{\alpha_1}$ and $h$, there is $\alpha_k \in A$ such that $\alpha_k^{hg_3^{-1}g_i} = \alpha_j$.  Since $\alpha_1 = \gamma \not\sim \alpha_k$, we have $\alpha_j = \alpha_k^{hg_3^{-1}g_i} \not\sim \alpha_1^{hg_3^{-1}g_i} = \alpha_i$.  Since $i,j$ were arbitrary, $A$ is a coclique.  Since $\Gamma[\alpha_1, \beta_1, \alpha_2, \beta_2] \cong C_4$, it immediately follows that $B$ is a coclique as well.
 
 Suppose now that $\gamma^h = \beta_3$.  Let $\gamma = \beta_1$, and let $A := \{\alpha_i : 1 \le i \le m \}$ and $B := \{\beta_i : 1 \le i \le m\}$.  The proof now proceeds as above.  By Lemma \ref{lem:vertstab2trans}, for each $i \ge 2$ there exists $g_i \in M_{\beta_1}$ such that $\beta_1^{g_i} = \beta_1$ and $\beta_2^{g_i} = \beta_i$.  Note that \[B_2 = (B_2 \cap A_1) \cup (B_2 \cap B_1),\]
 \[A_2 = (A_2 \cap A_1) \cup (A_2 \cap B_1),\]
 where $A_i$ and $B_i$ are defined as in the statement of Lemma \ref{lem:flipall}.  Since $\beta_2^h = \beta_2$, $B_2^h = B_2$ and $A_2^h = A_2$, and so either (i) $B^h = B$ and $A^h = A$ or (ii) $h$ swaps $(B_1 \cap A_2)$ and $(A_1 \cap A_2)$ and $h$ swaps $(B_1 \cap B_2)$ and $(A_1 \cap B_2)$.  However, $\beta_2 \in B_1 \cap B_2$, so we have $B^h = B$ and $A^h = A$.  Let $\beta_i, \beta_j \in B$, $i \neq j$.  Since $B$ is invariant under $M_{\beta_1}$ and $h$, there is $\beta_k \in B$ such that $\beta_k^{hg_3^{-1}g_i} = \beta_j$.  Since $\beta_1 = \gamma \sim \beta_k$, we have $\beta_j = \beta_k^{hg_3^{-1}g_i} \sim \beta_1^{hg_3^{-1}g_i} = \beta_i$.  Since $i,j$ were arbitrary, $B$ is a clique.  Since $\Gamma[\alpha_1, \beta_1, \alpha_2, \beta_2] \cong C_4$, it immediately follows that $A$ is a clique as well.
\end{proof}

\begin{lem}\label{lem:11case}
If $\Gamma[\alpha_{1},\beta_{1},\alpha_{2},\beta_{2}]\cong C_{4}$, then either $\Gamma=K_{m}\veebar K_{m}$ or $\Gamma=K_{m,m}$.
\end{lem}

\begin{proof}
This follows immediately from Lemma \ref{lem:partition} and a consideration of the degree of each vertex in the induced subgraph $\Gamma[\alpha_{1},\beta_{1},\alpha_{2},\beta_{2}]\cong C_{4}$.
\end{proof}

\subsection{The cases $\Gamma[\alpha_1, \beta_1, \alpha_2, \beta_2] \cong P_4$ and $\Gamma[\alpha_1, \beta_1, \alpha_2, \beta_2] \cong K_4 \backslash \{e\}$}

The two remaining cases are actually very closely related.  We begin with a helpful lemma.  

\begin{lem}
 \label{lem:regvsnonreg}
 Assume $\Gamma[\alpha_1, \beta_1, \alpha_2, \beta_2] \cong P_4$ or $\Gamma[\alpha_1, \beta_1, \alpha_2, \beta_2] \cong K_4 \backslash \{e\}$.  Either
\begin{itemize}
\item[(1)] $\Gamma$ is regular, or
\item[(2)] $\Gamma$ has two orbits of vertices: one orbit is a clique, the other a coclique.
\end{itemize}
\end{lem}

\begin{proof}
 We know that $G$ is transitive on $\M$, so $\Aut(\Gamma)$ has at most $2$ orbits of vertices. If $\Aut(\Gamma)$ is also transitive on $V(\Gamma)$, then $(1)$ holds. If not, $\Gamma$ has exactly $2$ orbits of vertices, and $\Aut(\Gamma)$ will be $2$-transitive on each of these orbits. Thus each orbit is either a clique or a coclique.  Both cannot be cliques, because otherwise $\Gamma$ would be regular. Both cannot be cocliques, because otherwise $\Gamma$ is not connected. So we are in case $(2)$.
\end{proof}

This allows us immediately to classify these graphs in the event that they are not regular.

\begin{lem}
 \label{lem:1021nonreg}
 If $\Gamma[\alpha_1, \beta_1, \alpha_2, \beta_2] \cong P_4$ and $\Gamma$ is not regular, then $\Gamma \cong K_m \veebar  \overline{K}_m$.  If we have $\Gamma[\alpha_1, \beta_1, \alpha_2, \beta_2] \cong K_4 \backslash \{e\}$ and $\Gamma$ is not regular, then $\Gamma \cong K_m \vee \overline{K}_m$.
\end{lem}

\begin{proof}
 This follows immediately from Lemma \ref{lem:regvsnonreg}.
\end{proof}

The remaining cases are when $\Gamma$ is regular.  This implies that $m$ is odd.  

\begin{lem}
 \label{lem:modd}
 Assume $\Gamma[\alpha_1, \beta_1, \alpha_2, \beta_2] \cong P_4$ or $\Gamma[\alpha_1, \beta_1, \alpha_2, \beta_2] \cong K_4 \backslash \{e\}$.  If $\Gamma$ is regular, then $m$ is odd.  Moreover, if $m = 2k+1$, then the degree of each vertex is $k+1$ if $\Gamma[\alpha_1, \beta_1, \alpha_2, \beta_2] \cong P_4$ and the degree of each vertex is $3k+1$ if $\Gamma[\alpha_1, \beta_1, \alpha_2, \beta_2] \cong K_4 \backslash \{e\}$.
\end{lem}

\begin{proof}
 Assume that $\Gamma[\alpha_1, \beta_1, \alpha_2, \beta_2] \cong P_4$.  For each $i \ge 2$, the vertices of $e_i$ contribute $0$ to the degree of one endpoint of $e_1$ and $1$ to the other, i.e., each $e_i$ for $i \ge 2$ contributes $1$ to the sum of the degree of $\alpha_1$ and the degree of $\beta_1$.  Since $\Gamma$ is regular,
 \[2 \cdot |\Gamma(\alpha_1)| = |\Gamma(\alpha_1)| + |\Gamma(\beta_1)| = 1 + 1 + (m-1).\]
 The result follows for $\Gamma[\alpha_1, \beta_1, \alpha_2, \beta_2] \cong P_4$. The proof is analogous in the case when we have $\Gamma[\alpha_1, \beta_1, \alpha_2, \beta_2] \cong K_4 \backslash \{e\}$.
\end{proof}

In fact, in these remaining cases when $\Gamma$ is regular, $\Gamma$ must be vertex transitive.

\begin{lem}
\label{lem:1021vertextrans}
 Assume $\Gamma[\alpha_1, \beta_1, \alpha_2, \beta_2] \cong P_4$ or $\Gamma[\alpha_1, \beta_1, \alpha_2, \beta_2] \cong K_4 \backslash \{e\}$.  If $\Gamma$ is regular, then $M$ is transitive on $V(\Gamma)$.
\end{lem}

\begin{proof}
 We know that $M$ is transitive on the edges of $\M$, so it suffices to show that there is $g_i \in \Aut(\Gamma)$ such that $\alpha_i^{g_i} = \beta_i$ for each $i$.  Assuming $\Gamma$ contains more than a single edge, it must contain at least three edges since $m$ is odd.  In each case we may choose three edges as follows:
 \begin{multicols}{2}
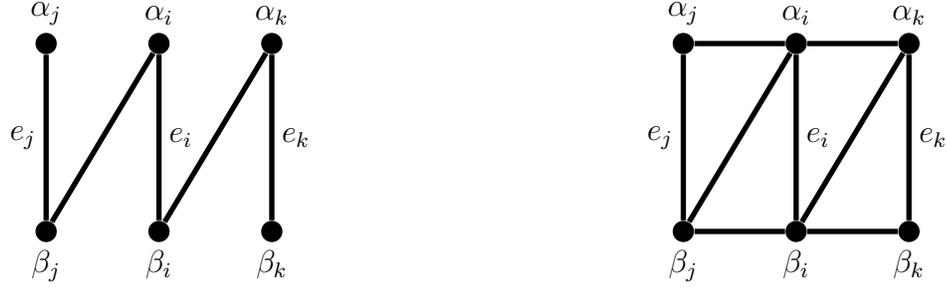
\begin{center}
\begin{tikzpicture}[scale=0.5]
\node[fill, shape=circle, label=above:$\alpha_{j}$] (aj) at (0,5) {};
\node[fill, shape=circle, label=below:$\beta_{j}$] (bj) at (0,0) {};
\node[fill, shape=circle, label=above:$\alpha_{i}$] (ai) at (3,5) {};
\node[fill, shape=circle, label=below:$\beta_{i}$] (bi) at (3,0) {};
\node[fill, shape=circle, label=above:$\alpha_{k}$] (ak) at (6,5) {};
\node[fill, shape=circle, label=below:$\beta_{k}$] (bk) at (6,0) {};

\draw[line width=2pt] (aj)--(bj) node[pos=.5, left] {$e_{j}$};
\draw[line width=2pt] (ai)--(bi) node[pos=.5, right] {$e_{i}$};
\draw[line width=2pt] (ak)--(bk) node[pos=.5, right] {$e_{k}$};
\draw[line width=2pt] (ai)--(bj);
\draw[line width=2pt] (ak)--(bi);
\end{tikzpicture}
\end{center}
\begin{center}
\begin{tikzpicture}[scale=0.5]
\node[fill, shape=circle, label=above:$\alpha_{j}$] (aj) at (0,5) {};
\node[fill, shape=circle, label=below:$\beta_{j}$] (bj) at (0,0) {};
\node[fill, shape=circle, label=above:$\alpha_{i}$] (ai) at (3,5) {};
\node[fill, shape=circle, label=below:$\beta_{i}$] (bi) at (3,0) {};
\node[fill, shape=circle, label=above:$\alpha_{k}$] (ak) at (6,5) {};
\node[fill, shape=circle, label=below:$\beta_{k}$] (bk) at (6,0) {};

\draw[line width=2pt] (aj)--(bj) node[pos=.5, left] {$e_{j}$};
\draw[line width=2pt] (ai)--(bi) node[pos=.5, right] {$e_{i}$};
\draw[line width=2pt] (ak)--(bk) node[pos=.5, right] {$e_{k}$};
\draw[line width=2pt] (ai)--(bj);
\draw[line width=2pt] (ak)--(bi);
\draw[line width=2pt] (aj)--(ai);
\draw[line width=2pt] (ai)--(ak);
\draw[line width=2pt] (bj)--(bi);
\draw[line width=2pt] (bi)--(bk);
\end{tikzpicture}
\end{center}
\end{multicols}

By the $2$-transitivity of $M$ on $\M$, there is $g \in M$ such that $e_i^g = e_i$ and $e_j^g = e_k$.  The $g_i$ that we seek is this $g$, and the result follows.
\end{proof}

We now show that there is a bijection between regular graphs in these two cases, i.e. that the two cases correspond.

\begin{lem}
 \label{lem:regbij}
 There exists a regular graph $\Gamma_0$ on $2m$ vertices with $\Gamma_0[\alpha_1, \beta_1, \alpha_2, \beta_2] \cong P_4$ if and only if there exists a regular graph $\Gamma_1$ on $2m$ vertices with $\Gamma_1[\alpha_1, \beta_1, \alpha_2, \beta_2] \cong K_4 \backslash \{e\}$, and there is a natural bijection between such graphs.  
\end{lem}

\begin{proof}
 Suppose we have such a graph $\Gamma_1$.  Note that $M$ is the setwise stabilizer of $\M$ in $\Aut(\Gamma_1)$, which is transitive on $V(\Gamma_1)$ but preserves the matching $\M$.  However, $M$ has (at least) two orbits on the edges of $\Gamma_1$: the edges of $\M$ and the edges not in $\M$.  The complement $\overline{\Gamma}_1$ also has $M$ as a group of automorphisms.  We define $\Gamma_0$ to be the graph with vertex set $V(\Gamma_1)$ and edge set $E(\overline{\Gamma}_1) \cup \M$.  The group $M$ is still $2$-transitive on a perfect matching in this case, but $\Gamma_0[\alpha_1, \beta_1, \alpha_2, \beta_2] \cong P_4$.  The proof in the other direction is analogous. 
\end{proof}

After considering Lemma \ref{lem:regbij}, there are really only three cases left.  We may assume that $\Gamma[\alpha_1, \beta_1, \alpha_2, \beta_2] \cong P_4$, and one of the following holds: (i) $\Aut(\Gamma)$ is primitive on $V(\Gamma)$, (ii) $\Gamma$ is bipartite, or (iii) $\M$ itself is a system of imprimitivity.  (Any other system of imprimitivity is ruled out by the $2$-transitivity of $M$ on $\M$.)

\subsection{The case where $\Gamma[\alpha_1, \beta_1, \alpha_2, \beta_2] \cong P_4$ and $G=\Aut(\Gamma)$ is primitive on vertices}\


\begin{lem}
 \label{lem:edgetrans}
 Assume $\Gamma[\alpha_1, \beta_1, \alpha_2, \beta_2] \cong P_4$ and $G = \Aut(\Gamma)$ is primitive on vertices.  Then $\Gamma$ is a $(G,2)$-arc-transitive graph.
\end{lem}

\begin{proof}
 Since $G$ is primitive on $V(\Gamma)$, $G_{\alpha_1}$ is a maximal subgroup of $G$.  On the other hand, since there is $g \in M$ such that $\alpha_1^g = \beta_1$, $\beta_1^g = \alpha_1$ (see Lemma \ref{lem:1021vertextrans}), we have $M_{\alpha_1} < M_{e_1} < M \le G$, and so $M_{\alpha_1}$ is not a maximal subgroup of $M$.  Thus $M < G$.
 
 By the $2$-transitivity of $M$ on $\M$, $M$ has two orbits on $E(\Gamma)$: $\M$ and $E(\Gamma) \backslash \M$.  Since $M < G$, there is $h \in G \backslash M$, i.e., there is an automorphism that does not preserve $\M$.  This implies that $h$ takes an edge in $\M$ to an edge in $E(\Gamma) \backslash \M$, and so $G$ is transitive on $E(\Gamma)$.
 
 Finally, we note that (i) $\Gamma$ is $G$-vertex-transitive, (ii) $\Gamma$ is $G$-edge-transitive, (iii) there is an element sending the arc $(\alpha_1, \beta_1)$ to the arc $(\beta_1, \alpha_1)$, and (iv) $G_{\alpha_1\beta_1}$ is transitive on $\Gamma(\alpha_1) \backslash \{\beta_1\}$, which implies that $\Gamma$ is a $(G,2)$-arc-transitive graph.
\end{proof}

Consider the labeling of the vertices as in Figure \ref{fig:01prim}.  If we define $D_i(\gamma) := \{\delta \in V(\Gamma) : d(\gamma, \delta) = i \}$, i.e., if $D_i(\gamma)$ is the set of vertices at distance $i$ from the vertex $\gamma$, we can guarantee the distance of all vertices in the graph from $\alpha_1$ except for the set $X$; all we know is that $X \subseteq D_2(\alpha_1) \cup D_3(\alpha_1)$.

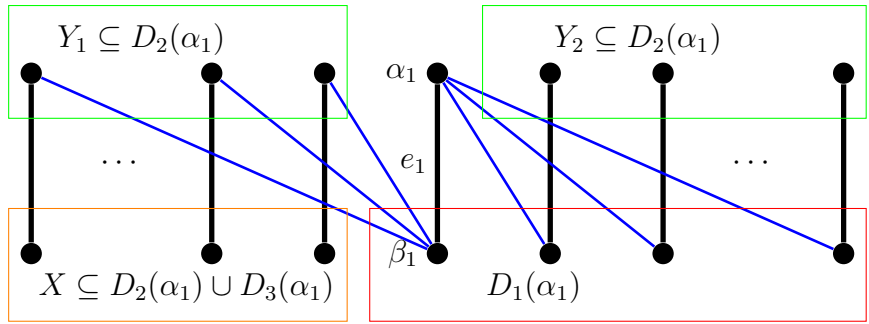
\begin{figure}[h!]
\begin{center}
\begin{tikzpicture}[scale=0.3]
\node[fill, shape=circle] (an) at (-18,8) {};
\node[fill, shape=circle] (bn) at (-18,0) {};
\node[label=above:$\cdots$] at (-14,3) {};
\node[fill, shape=circle] (a5) at (-10,8) {};
\node[fill, shape=circle] (b5) at (-10,0) {};
\node[fill, shape=circle] (a4) at (-5,8) {};
\node[fill, shape=circle] (b4) at (-5,0) {};

\node[fill, shape=circle, label=left:$\alpha_{1}$] (a1) at (0,8) {};
\node[fill, shape=circle, label=left:$\beta_{1}$] (b1) at (0,0) {};

\node[fill, shape=circle] (a2) at (5,8) {};
\node[fill, shape=circle] (b2) at (5,0) {};
\node[fill, shape=circle] (a3) at (10,8) {};
\node[fill, shape=circle] (b3) at (10,0) {};
\node[label=above:$\cdots$] at (14,3) {};
\node[fill, shape=circle] (am) at (18,8) {};
\node[fill, shape=circle] (bm) at (18,0) {};

\draw[line width=2pt] (a1)--(b1) node[pos=.5, left] {$e_{1}$};
\draw[line width=2pt] (a2)--(b2) node[pos=.5, right] {};
\draw[line width=2pt] (a3)--(b3) node[pos=.5, right] {};
\draw[line width=2pt] (am)--(bm) node[pos=.5, right] {};
\draw[line width=2pt] (a4)--(b4) node[pos=.5, right] {};
\draw[line width=2pt] (a5)--(b5) node[pos=.5, right] {};
\draw[line width=2pt] (an)--(bn) node[pos=.5, right] {};
\draw[line width=1pt, color=blue] (a1)--(b2) node[pos=.5, right] {};
\draw[line width=1pt, color=blue] (a1)--(b3) node[pos=.5, right] {};
\draw[line width=1pt, color=blue] (a1)--(bm) node[pos=.5, right] {};
\draw[line width=1pt, color=blue] (a4)--(b1) node[pos=.5, right] {};
\draw[line width=1pt, color=blue] (a5)--(b1) node[pos=.5, right] {};
\draw[line width=1pt, color=blue] (an)--(b1) node[pos=.5, right] {};
\draw[color=red](-3,2) rectangle (19,-3) {};
\draw[color=orange](-4,2) rectangle (-19,-3) {};
\draw[color=green](-4,11) rectangle (-19,6) {};
\draw[color=green](2,11) rectangle (19,6) {};
\node[text width=3.5cm] at (8,-1.5) {$D_1(\alpha_1)$};
\node[text width=3.5cm] at (11,9.5) {$Y_2 \subseteq D_2(\alpha_1)$};
\node[text width=3.5cm] at (-11,9.5) {$Y_1 \subseteq D_2(\alpha_{1})$};
\node[text width=4cm] at (-11,-1.5) {$X \subseteq D_2(\alpha_1) \cup D_3(\alpha_1)$};
\end{tikzpicture}
\caption{Labeling of $\Gamma$ when $\Gamma[\alpha_1, \beta_1, \alpha_2, \beta_2] \cong P_4$ and $G$ is primitive on vertices.}
\label{fig:01prim}
\end{center}
\end{figure}

\begin{lem}
 \label{lem:Xdist3}
 Assume $\Gamma[\alpha_1, \beta_1, \alpha_2, \beta_2] \cong P_4$, $G = \Aut(\Gamma)$ is primitive on vertices, and the subset $X$ is as defined above.  Then $X \cap D_2(\alpha_1) \neq \varnothing$.
\end{lem}

\begin{proof}
 Suppose that $X \cap D_2(\alpha_1) = \varnothing$, that is, $X = D_3(\alpha_1)$.  By Lemma \ref{lem:edgetrans} and the fact that $\M$ is a $2$-transitive perfect matching, $\Gamma$ is distance-transitive with diameter $3$.  We will show that, if $\gamma, \delta \in X$, then $\gamma \not\sim \delta$.  Indeed, suppose $\gamma \in D_3(\alpha_1) = X$.  This means that $d(\beta_1, \gamma) = 2$.  Since $X \cap D_2(\alpha_1) = \varnothing$, there are no edges from $D_1(\alpha_1)$ to $X$.  Similarly, since $\Gamma$ is vertex-transitive, there are no edges from $D_1(\beta_1) = Y_1 \cup \{\alpha_1\}$ to $D_3(\beta_1) = Y_2$ (see Figure \ref{fig:01prim}).  Hence, if $\delta \in D_1(\beta_1)$, $\delta$ has no neighbors in $Y_2$.  Since $\Gamma$ is distance-transitive, this means that no vertex in $D_2(\alpha_1)$ has any neighbors in $D_2(\alpha_1)$, i.e., all edges in $\Gamma$ are from $A = \{\alpha_1\} \cup Y_1 \cup Y_2$ to $X \cup D_1(\alpha_1)$.  However, this means that $\Gamma$ is bipartite, in contradiction to $\Gamma$ being vertex-primitive.  Therefore, $X \cap D_2(\alpha_1) \neq \varnothing$.
\end{proof}

\begin{lem}
 \label{lem:01prim}
 Assume $\Gamma[\alpha_1, \beta_1, \alpha_2, \beta_2] \cong P_4$ and $G = \Aut(\Gamma)$ is primitive on vertices.  Then $\Gamma$ is isomorphic to the Petersen graph.
\end{lem}

\begin{proof}
 We again assume that vertices are labeled as in Figure \ref{fig:01prim}.  By the $2$-transitivity of $M$ on $\M$, $M_{\alpha_1}$ is transitive on $X$, and so $X \subseteq D_2(\alpha_1)$.  Hence $\Gamma$ has diameter $2$, and, by Lemma \ref{lem:edgetrans}, $\Gamma$ is a distance-transitive, diameter $2$, triangle-free strongly regular graph.  (These are known as \textit{rank 3 graphs} since, for any vertex $\alpha \in V(\Gamma)$, the stabilizer of $\alpha$ is a primitive group of rank $3$ on vertices.)  We note that $\Gamma$ is a $(4k+2, k+1, 0, \mu)$-strongly regular graph.
 
 By the classic equation relating the parameters (see \cite{DRG}),
\[(k+1)k  = [(4k+2)-(k+1)-1]\mu, \]
and so $\mu =(k+1)/3$.

The eigenvalues of the adjacency matrix for this graph and their multiplicities are known (again, see \cite{DRG}). There are three eigenvalues: $k+1$, with multiplicity one, and two others. The multiplicities of these other two eigenvalues are
\begin{align*}
& \frac{1}{2}\left((4k+1)\pm\frac{(4k+1)\left(\frac{k+1}{3}\right)-(2k+2)}{\sqrt{\left(\frac{k+1}{3}\right)^{2}+8\frac{k+1}{3}}}\right)\\
& =  \frac{1}{2}\left((4k+1)\pm\frac{4k^{2}-k-5}{\sqrt{(k+1)(k+25)}}\right) \in \Z.\\
\end{align*}

This implies that \[\frac{(4k^{2}-k-5)^{2}}{(k+1)(k+25)}=16k^{2}-424k+10585-\frac{264000}{k+25}\] is a perfect square. The last term allows us, via factoring, to come up with a list of values of $k$ to check, which yields \[k=2 \text{ or } k=24.\] But, together with $\mu=\frac{k+1}{3}\in\Z$, we rule out $k=24$, so the only graph in this case is strongly regular with parameters $(10,3,0,1)$ (corresponding to $k=2$), which is the Petersen graph.  It can be verified that the Petersen graph has a $2$-transitive perfect matching by direct inspection.  For instance, if the vertices of the Petersen graph $\mathcal{P}$ are represented as subsets of size two of $\{1,2,3,4,5\}$, then $\Aut(\mathcal{P}) = S_5$ is $2$-transitive on the matching 
\[ \M = \left\{ \{\{1,2\},\{3,4\}\}, \{\{3,5\},\{2,4\}\}, \{\{1,4\},\{2,5\}\}, \{\{2,3\},\{1,5\}\}, \{\{4,5\},\{1,3\}\} \right\}.\]
\end{proof}

\subsection{The case where $\Gamma[\alpha_1, \beta_1, \alpha_2, \beta_2] \cong P_4$ and $G=\Aut(\Gamma)$ is imprimitive on vertices}\


\begin{lem}
 \label{lem:01imprim}
 Assume $\Gamma[\alpha_1, \beta_1, \alpha_2, \beta_2] \cong P_4$ and that $\Pi = \{\{\alpha_i, \beta_i\} : 1 \le i \le m\}$ is a system of imprimitivity on $V(\Gamma)$.  Then $m = p^f$, where $p$ is a prime and $p^f \equiv 3 \pmod 4$, and $\Gamma$ is isomorphic to the incidence graph of the Paley symmetric $2$-design over $\GF(p^f)$.
\end{lem}

\begin{proof}
 Suppose $\Pi = \{\{\alpha_i, \beta_i\} : 1 \le i \le m\}$ is a system of imprimitivity on $V(\Gamma)$.  This implies that $G = M$.  We remove the edge orbit $\M$ from $\Gamma$ to create a new graph $\Gamma'$; since $G = M$, $\M$ is an orbit of the edges of $\Gamma$ under $G$, and $G$ still acts $2$-transitively on the system of imprimitivity $\Pi$.  However, each block in $\Pi$ is now an independent set.  The quotient graph $\Gamma'_\Pi$ will be the complete graph $K_{m}$, and there is exactly one edge between any two blocks in $\Gamma'$.  By the $2$-transitivity of $M$ on $\Pi$, $\Gamma'$ is $M$-arc-transitive.  Hence $\Gamma'$ is a \textit{symmetric spread} of the complete graph $K_m$ (see \cite{SymmSpreads}).  By inspection of \cite[Tables 1, 2]{SymmSpreads}, the only possibility for $\Gamma$ is the incidence graph of the Paley symmetric $2$-design over $\GF(p^f)$.  Moreover, if $\Gamma$ is such a graph, then $V(\Gamma) = \GF(p^f) \times \{0,1\}$, and $\Aut(\Gamma)$ acts $2$-transitively on each copy of $\GF(p^f)$ (simultaneously).  Hence the matching $\{ \{(x,0), (x,1)\} : x \in \GF(p^f)\}$ is a $2$-transitive perfect matching.
\end{proof}

Our final case is when $\Gamma[\alpha_1, \beta_1, \alpha_2, \beta_2] \cong P_4$ and $\Gamma$ is bipartite.

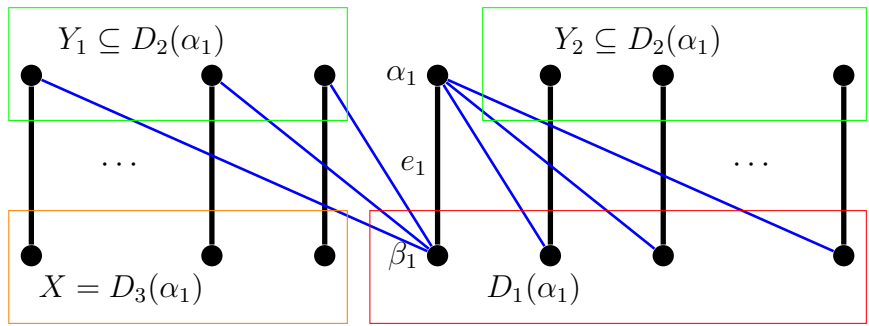
\begin{figure}[h!]
\begin{center}
\begin{tikzpicture}[scale=0.3]
\node[fill, shape=circle] (an) at (-18,8) {};
\node[fill, shape=circle] (bn) at (-18,0) {};
\node[label=above:$\cdots$] at (-14,3) {};
\node[fill, shape=circle] (a5) at (-10,8) {};
\node[fill, shape=circle] (b5) at (-10,0) {};
\node[fill, shape=circle] (a4) at (-5,8) {};
\node[fill, shape=circle] (b4) at (-5,0) {};

\node[fill, shape=circle, label=left:$\alpha_{1}$] (a1) at (0,8) {};
\node[fill, shape=circle, label=left:$\beta_{1}$] (b1) at (0,0) {};

\node[fill, shape=circle] (a2) at (5,8) {};
\node[fill, shape=circle] (b2) at (5,0) {};
\node[fill, shape=circle] (a3) at (10,8) {};
\node[fill, shape=circle] (b3) at (10,0) {};
\node[label=above:$\cdots$] at (14,3) {};
\node[fill, shape=circle] (am) at (18,8) {};
\node[fill, shape=circle] (bm) at (18,0) {};

\draw[line width=2pt] (a1)--(b1) node[pos=.5, left] {$e_{1}$};
\draw[line width=2pt] (a2)--(b2) node[pos=.5, right] {};
\draw[line width=2pt] (a3)--(b3) node[pos=.5, right] {};
\draw[line width=2pt] (am)--(bm) node[pos=.5, right] {};
\draw[line width=2pt] (a4)--(b4) node[pos=.5, right] {};
\draw[line width=2pt] (a5)--(b5) node[pos=.5, right] {};
\draw[line width=2pt] (an)--(bn) node[pos=.5, right] {};
\draw[line width=1pt, color=blue] (a1)--(b2) node[pos=.5, right] {};
\draw[line width=1pt, color=blue] (a1)--(b3) node[pos=.5, right] {};
\draw[line width=1pt, color=blue] (a1)--(bm) node[pos=.5, right] {};
\draw[line width=1pt, color=blue] (a4)--(b1) node[pos=.5, right] {};
\draw[line width=1pt, color=blue] (a5)--(b1) node[pos=.5, right] {};
\draw[line width=1pt, color=blue] (an)--(b1) node[pos=.5, right] {};
\draw[color=red](-3,2) rectangle (19,-3) {};
\draw[color=orange](-4,2) rectangle (-19,-3) {};
\draw[color=green](-4,11) rectangle (-19,6) {};
\draw[color=green](2,11) rectangle (19,6) {};
\node[text width=3.5cm] at (8,-1.5) {$D_1(\alpha_1)$};
\node[text width=3.5cm] at (11,9.5) {$Y_2 \subseteq D_2(\alpha_1)$};
\node[text width=3.5cm] at (-11,9.5) {$Y_1 \subseteq D_2(\alpha_{1})$};
\node[text width=4cm] at (-11,-1.5) {$X = D_3(\alpha_1)$};
\end{tikzpicture}
\caption{Labeling of $\Gamma$ when $\Gamma[\alpha_1, \beta_1, \alpha_2, \beta_2] \cong P_4$, $M < G$, $\Gamma$ bipartite}
\label{fig:01bip}
\end{center}
\end{figure}

\begin{lem}
 \label{lem:01bipartite}
 Assume $\Gamma[\alpha_1, \beta_1, \alpha_2, \beta_2] \cong P_4$ and that $\Gamma$ is bipartite.  Then $m = p^f$, where $p$ is a prime and $p^f \equiv 3 \pmod 4$, and $\Gamma$ is isomorphic to the incidence graph of the Paley symmetric $2$-design over $\GF(p^f)$.
\end{lem}

\begin{proof}
 Assume $\Gamma[\alpha_1, \beta_1, \alpha_2, \beta_2] \cong P_4$ and that $\Gamma$ is bipartite.  If $G := \Aut(\Gamma) = M$, then this case has been resolved by Lemma \ref{lem:01imprim}.  Hence we may assume that $M < G$.  By the $2$-transitivity of $M$ on $\M$, $M$ has two orbits on $E(\Gamma)$: $\M$ and $E(\Gamma) \backslash \M$.  Since $M < G$, there is $h \in G \backslash M$, i.e., there is an automorphism that does not preserve $\M$.  This implies that $h$ takes an edge in $\M$ to an edge in $E(\Gamma) \backslash \M$, and so $G$ is transitive on $E(\Gamma)$.
 
 Since (i) $G$ is transitive on the edges of $\Gamma$, (ii) $G$ is transitive on the vertices of $\Gamma$, (iii) there is an element sending the arc $(\alpha_1, \beta_1)$ to the arc $(\beta_1, \alpha_1)$ by Lemma \ref{lem:1021vertextrans}, and (iv) $G_{\alpha_1\beta_1}$ is transitive on $\Gamma(\alpha) \backslash \{\beta_1\}$, we have that $\Gamma$ is a $(G,2)$-arc-transitive graph.  Since $\Gamma$ is bipartite, using the labeling of Figure \ref{fig:01bip}, we have $D_2(\alpha_1) = Y_1 \cup Y_2$ and $D_3(\alpha_1) = X$.  Since $\Gamma$ is $(G,2)$-arc-transitive and $M_{\alpha_1}$ is transitive on $X$, $\Gamma$ is a distance-transitive graph of diameter $3$.  By \cite[Theorem 5.10.3]{godsilroyle}, $\Gamma$ is the incidence graph of a symmetric $2$-design.  The points of the design are represented by one of the biparts of $\Gamma$.  The stabilizer of a point (i.e., of $\alpha_1$, say) has at most three orbits on points: (i) $\{\alpha_1\}$, (ii) the set of all points incident with ``block'' $\beta_1$, and (iii) set of all points not incident with ``block'' $\beta_1$.  This means that $\Gamma$ is the incidence graph of a rank 2 or 3 symmetric $2$-design.  Such symmetric $2$-designs have been classified \cite{dempwolffclassify1, dempwolffclassify2, kantorclassify}.  The only possibilities, other than the Paley symmetric $2$-designs, are: the Hadamard design with $11$ points where each point is incident with $5$ blocks, which gives the same incidence graph as the Paley symmetric $2$-design on $11$ points; the design with $35$ points where each point is incident with exactly $17$ blocks, which is ruled out since the only $2$-transitive groups on $35$ points are $A_{35}$ and $S_{35}$, which are not involved in the automorphism group of this design (the unique minimal normal subgroup of the automorphism group of this design is isomorphic to $A_8$); and the design with $15$ points where each point is contained in exactly $7$ blocks.  In this last case, the unique minimal normal subgroup of the automorphism group of the design is isomorphic to $A_6$.  While $A_6$ has a rank 3 action on $15$ points, the stabilizer of a point in this action has orbits of size $1$, $6$, and $8$.  However, if the incidence graph of this design had a $2$-transitive perfect matching, then the stabilizer of a point would have orbits of size $1$, $7$, and $7$.  Therefore, the only such graphs $\Gamma$ with $\Gamma[\alpha_1, \beta_1, \alpha_2, \beta_2] \cong P_4$ and $\Gamma$ bipartite are isomorphic to incidence graphs of Paley symmetric $2$-designs. 
\end{proof}

We are now ready to complete the proof of Theorem \ref{thm:2transperfect}.

\begin{proof}[Proof of Theorem \ref{thm:2transperfect}]
The result follows from Lemmas \ref{lem:div}, \ref{lem:22case}, \ref{lem:20case}, \ref{lem:11case}, \ref{lem:1021nonreg}, \ref{lem:regbij}, \ref{lem:01prim}, \ref{lem:01imprim}, and \ref{lem:01bipartite}. 
\end{proof}

Finally, we prove Corollary \ref{cor:mpermperfect}.

\begin{proof}[Proof of Corollary \ref{cor:mpermperfect}]
The result follows from Theorem \ref{thm:2transperfect} and noting which graphs in cases (3) and (4) have an induced symmetric group on the matching.  Since the group acting on the matching in each of (3) and (4) has a minimal normal subgroup that is elementary abelian and acts regularly on an odd number of edges, we conclude that the only option in cases (3) and (4) is when $m = 3$.  The result follows. 
\end{proof}

\noindent\textsc{Acknowledgements.}  
The authors wish to thank Thomas Zaslavsky for his numerous editorial suggestions and comments on earlier versions of this paper and the anonymous referees for their helpful reports.

\bibliographystyle{plain}
\bibliography{HighlyTransitiveMatchings.bib}

\end{document}